\documentclass[12pt, reqno, a4paper]{amsart}

\usepackage{amssymb}
\usepackage{hyperref}
\usepackage[pdftex]{lscape}

\usepackage[english]{babel}

\theoremstyle{plain}
\newtheorem{theorem}{Theorem}[section]
\newtheorem{lemma}[theorem]{Lemma}

\newtheorem{conjecture}[theorem]{Conjecture}

\theoremstyle{remark}
\newtheorem{remark}[theorem]{Remark}

\theoremstyle{definition}
\newtheorem{example}[theorem]{Example}

\numberwithin{equation}{section}

\DeclareMathOperator{\Id}{Id}

\DeclareMathOperator{\id}{id}
\DeclareMathOperator{\ch}{char}

\DeclareMathOperator{\ad}{ad}

\DeclareMathOperator{\Aut}{Aut}

\DeclareMathOperator{\sign}{sign}
\DeclareMathOperator{\Alt}{Alt}

\DeclareMathOperator{\Hom}{Hom}

\DeclareMathOperator{\PIexp}{PIexp}

\begin{document}

\title{Lie algebras simple with respect to a Taft algebra action}

\author{A.\,S.~Gordienko}
\address{Vrije Universiteit Brussel, Belgium}
\email{alexey.gordienko@vub.ac.be}


\begin{abstract} 
We classify finite dimensional $H_{m^2}(\zeta)$-simple $H_{m^2}(\zeta)$-module Lie algebras $L$ over an algebraically closed field of characteristic $0$ where $H_{m^2}(\zeta)$ is the $m$th Taft algebra.
As an application, we show that despite the fact that $L$ can be non-semisimple in ordinary sense, 
$\lim_{n\to\infty}\sqrt[n]{c_n^{H_{m^2}(\zeta)}(L)} = \dim L$ where $c_n^{H_{m^2}(\zeta)}(L)$ is the codimension sequence of polynomial $H_{m^2}(\zeta)$-identities of $L$. In particular, the analog of Amitsur's conjecture holds for $c_n^{H_{m^2}(\zeta)}(L)$. 
\end{abstract}


\thanks{Supported by Fonds Wetenschappelijk Onderzoek~--- Vlaanderen post doctoral fellowship (Belgium).}

\maketitle

\section{Introduction}

An $H_{m^2}(\zeta)$-module algebra, where $\zeta$ is a primitive $m$th root of unity, $m\in\mathbb N$, $m\geqslant 2$,
 is an algebra $A$ endowed with an automorphism $c$ and a skew-derivation $v$
such that $vc=\zeta cv$, $c^m=\id_{A}$, and $v^m=0$.
In particular, $A$ is a $\mathbb Z_m$-graded algebra where $\mathbb Z_m=\mathbb Z / m \mathbb Z$
is the cyclic group of order $m$ (see Remark~\ref{RemarkTaftZmGrading}).
$H_{m^2}(\zeta)$-module algebras provide, probably, easiest examples of $H$-module algebras for a non-semisimple Hopf algebra $H$. 
The study of $H$-module algebras can be considered as the next logical step
after the investigation of graded algebras, which have been studied extensively (see e.g.~\cite{BahturinZaicevSegal, ElduqueKochetov, NastasescuVanOystaeyen}).
In the context of polynomial identities, $H$-module Lie algebras were considered in~\cite{LinchenkoPILie}.
It is worth to notice that $H_{m^2}(\zeta)$-module algebras can have a structure quite different from the structure of $H$-module algebras for a semisimple Hopf algebra $H$. For example, an $H_{m^2}(\zeta)$-simple (i.e. not containing non-trivial $H_{m^2}(\zeta)$-invariant ideals) algebra can have a non-trivial radical.

Finite dimensional associative $H_{m^2}(\zeta)$-module algebras that contain no nonzero nilpotent elements were classified in~\cite[Theorem~2.5]{MontgomerySchneider}. Exact module categories over the category $\mathrm{Rep}(H_{m^2}(\zeta))$ were studied in~\cite[Theorem~4.10]{EtingofOstrik}.
Finite dimensional associative $H_{m^2}(\zeta)$-simple $H_{m^2}(\zeta)$-module algebras were classified in~\cite{ASGordienko11, ASGordienko12}.
$H_{m^2}(\zeta)$-actions on path algebras of quivers were studied in~\cite{KinserWalton}.

In Section~\ref{SectionClassTaftSLieSS} we construct semisimple $H_{m^2}(\zeta)$-simple Lie algebras
$L_\alpha(B)$
where $B$ is a finite dimensional simple Lie algebra and $\alpha \in F$ where $F$ is the base field. In Theorems~\ref{TheoremTaftSimpleSemiSimpleLieClassify} and~\ref{TheoremTaftSimpleSimpleLieClassify}
we prove that if $F$ is algebraically closed of characteristic $0$, then
each finite dimensional $H_{m^2}(\zeta)$-simple $H_{m^2}(\zeta)$-module Lie algebra
semisimple in ordinary sense
either has the zero $v$-action or is isomorphic as an $H_{m^2}(\zeta)$-module Lie algebra to one of the Lie algebras $L_\alpha(B)$. In Theorem~\ref{TheoremTaftSimpleSemiSimpleLieIso}
we show that $L_{\alpha_1}(B_1) \cong L_{\alpha_2}(B_2)$
if and only if $B_1 \cong B_2$ and $\alpha_2 = \zeta^k \alpha_1$ for some $k\in\mathbb Z$.

In order to exclude the case of simple $H_{m^2}(\zeta)$-simple Lie algebras, which is done in Theorem~\ref{TheoremTaftSimpleSimpleLieClassify}, and treat the case of non-semisimple $H_{m^2}(\zeta)$-simple Lie algebras, in Section~\ref{SectionClassTaftSLieSimple} we introduce $H_{m^2}(\zeta)$-simple Lie algebras $L(B,\gamma)$ where $\gamma \in F$ and $B$ is a simple Lie algebra. In Theorem~\ref{TheoremTaftSimpleLieEquivDef} we show that 
$L(B, \frac{1}{\alpha^m (1-\zeta)^m}) \cong L_\alpha(B)$ 
as $H_{m^2}(\zeta)$-module Lie algebras for $\alpha\ne 0$. In Theorem~\ref{TheoremTaftSimpleNonSemiSimpleLieClassify} we prove that every non-semisimple 
$H_{m^2}(\zeta)$-simple Lie algebra is isomorphic to one of the Lie algebras $L(B,0)$.
It turns out that the nilpotent radical of each $H_{m^2}(\zeta)$-simple Lie algebra coincides with its
solvable radical.

Although the classification of Lie algebras
is in some sense parallel to the associative case, the Lie case requires different techniques.
Furthermore, anti-commutativity of the commutator
in a Lie algebra is a strong restriction on the $H_{m^2}(\zeta)$-action
and we get much less possible parameters to describe $H_{m^2}(\zeta)$-simple Lie algebras.
In addition, every finite dimensional $H_{m^2}(\zeta)$-module Lie algebra
simple in the ordinary sense, is just a $\mathbb Z_m$-graded Lie algebra
with the zero skew-derivation (Theorem~\ref{TheoremTaftSimpleSimpleLieClassify}), which is in contrast to the associative case~\cite[Theorem~1]{ASGordienko12}.

In Section~\ref{SectionHPI-expTaftSLie} we apply the results obtained to codimensions of polynomial $H$-identities.

The codimension sequence $c_n(A)$ of polynomial identities of an algebra $A$ is an important numerical invariant of $A$.
It turns out that the asymptotic behaviour of $c_n(A)$ is tightly related to the structure of $A$~\cite{ZaiGia, ZaiLie}.

In 1980s, S.\,A.~Amitsur conjectured that if an associative algebra $A$
over a field of characteristic $0$ satisfies a nontrivial
polynomial identity, then there exists a \textit{PI-exponent} $\lim_{n\to\infty}\sqrt[n]{c_n(A)} \in \mathbb Z_+$. The original Amitsur conjecture was proved by A.~Giambruno and M.\,V.~Zaicev~\cite{ZaiGiaAmitsur} in 1999. Its analog for finite dimensional Lie algebras was proved by M.\,V.~Zaicev~\cite{ZaiLie}
in 2002.

In general, the analog of Amitsur's conjecture for infinite dimensional Lie algebras is wrong. First, the codimension growth can be overexponential~\cite{Volichenko}. Second, the exponent of the codimension growth can be non-integer~\cite{ZaiMishchFracPI, VerZaiMishch}. 
It is still unknown whether there exist Lie algebras $L$ with $$\mathop{\underline\lim}_{n\to\infty}\sqrt[n]{c_n(L)}\ne\mathop{\overline\lim}_{n\to\infty}\sqrt[n]{c_n(L)}.$$

Algebras endowed with an additional structure, e.g. grading, action of a group $G$, a Lie algebra or a Hopf algebra $H$, find their applications in many areas of mathematics and physics.
 For algebras with an additional structure, it is natural to consider the corresponding polynomial identities, namely, graded, differential, $G$- and $H$-identities.

 Polynomial $H$-identities have proved to be an important tool
in the study of graded polynomial identities in graded algebras.
In fact, they play a crucial role in the proof of the existence of an integer graded PI-exponent for an arbitrary group graded finite dimensional Lie algebra
~\cite[Theorem~1]{ASGordienko5}.
In~\cite[Theorem~7]{ASGordienko5} the author proved 
that for every finite dimensional semisimple Hopf algebra $H$
and every finite dimensional $H$-module Lie algebra
there exists integer $\PIexp^H(L):=\lim_{n\to\infty}\sqrt[n]{c_n^H(L)}$
where $c_n^H(L)$ is the codimension sequence of polynomial $H$-identities
of $L$, i.e. the analog of Amitsur's conjecture holds for $c_n^H(L)$.
Later the analog of Amitsur's conjecture was proved for some other classes of $H$-module Lie algebras where the solvable and nilpotent radicals
were still $H$-invariant~\cite{ASGordienko7}.

We believe that the analog of Amitsur's conjecture for $H$-module Lie algebras is true in the following form which belongs to Yu.\,A.~Bahturin:

\begin{conjecture}\label{ConjectureLieAmitsurBahturin} Let $L$ be a finite dimensional $H$-module Lie algebra for a Hopf algebra $H$ over a field of characteristic $0$. Then there exists
an integer $\PIexp^H(L):=\lim\limits_{n\to\infty}
 \sqrt[n]{c^H_n(L)}$.
\end{conjecture}

In Theorem~\ref{TheoremTaftSimpleLieHPIexpExists}
we show this in the case when $L$ is an $H_{m^2}(\zeta)$-simple 
$H_{m^2}(\zeta)$-module Lie algebra over an algebraically closed field of characteristic $0$.
In this case we have $\lim_{n\to\infty}\sqrt[n]{c_n^{H_{m^2}(\zeta)}(L)} = \dim L$.

\section{$H$-module algebras}\label{SectionHmodLie}

 A (not necessarily associative) algebra $A$
over a field $F$
is a \textit{(left) $H$-module algebra}
for some Hopf algebra $H$
if $L$ is a (left) $H$-module such that
$h(ab)=(h_{(1)}a)(h_{(2)}b)$
for all $h \in H$, $a,b \in A$. Here we use Sweedler's notation
$\Delta h = h_{(1)} \otimes h_{(2)}$ where $\Delta$ is the comultiplication
in $H$.
We refer the reader to~\cite{Danara, Montgomery, Sweedler}
   for an account
  of Hopf algebras and algebras with Hopf algebra actions.
  
  In the current article we study $H$-module Lie algebras $L$. The product of two elements $a,b\in L$  is denoted by $[a,b]$.
We say that $L$ is \textit{$H$-simple} if $[L,L]\ne 0$ and $L$ has no non-trivial $H$-invariant ideals.

  \begin{example}
  If $G$ is a group and $F$ is a field, then the group algebra $FG$ is a Hopf algebra
  where $\Delta(g)=g\otimes g$, $\varepsilon(g)=1$, and $S(g)=g^{-1}$
  for all $g\in G$. If $G$ is acting on a Lie algebra $L$ by automorphisms,
  then the $G$-action can be extended by linearity to an $FG$-action
  such that $L$ is an $FG$-module Lie algebra.
  \end{example}
  \begin{example}
  If $\mathfrak g$ is a Lie algebra over a field $F$, then its universal enveloping
  algebra $U(\mathfrak g)$ is a Hopf algebra where $\Delta(a)=1\otimes a + a \otimes 1$,
  $S(a)=-a$, $\varepsilon(a)=0$ for all $a\in\mathfrak g$. (The maps $\Delta \colon U(\mathfrak g)
  \to U(\mathfrak g) \otimes U(\mathfrak g)$ and $\varepsilon \colon U(\mathfrak g) \to F$ are extended from $\mathfrak g$ as homomorphisms of algebras with $1$
  and the map $S \colon U(\mathfrak g) \to U(\mathfrak g)$ is extended as an anti-homomorphism  of algebras with $1$.) If $\mathfrak g$ is acting on a Lie algebra $L$ by derivations, then the $\mathfrak g$-action
  can be naturally extended to a $U(\mathfrak g)$-action such that $L$ is a $U(\mathfrak g)$-module Lie algebra.
  \end{example}
  
Let $G$ be a group. A Lie algebra $L=\bigoplus_{g\in G} L^{(g)}$ (direct sum of subspaces) is \textit{$G$-graded} if  $[L^{(g)},L^{(h)}] \subseteq L^{(gh)}$ for all $g,h\in T$. A subspace $V$ of $L$ is \textit{graded} (or \textit{homogeneous}) if $V=\bigoplus_{g\in G} V \cap L^{(g)}$.

Consider the vector space $(FG)^*$ dual to $FG$. Then $(FG)^*$ is an associative algebra with the multiplication defined
by $(hw)(g)=h(g)w(g)$ for $h,w \in (FG)^*$ and $g\in G$. The identity element 
is defined by $1_{(FG)^*}(g)=1$ for all $g\in G$. In other words, $(FG)^*$ is the algebra dual to the coalgebra $FG$.

If $L=\bigoplus_{g\in G} L^{(g)}$ is a $G$-graded Lie algebra, then we have the following natural $(FG)^*$-action on $L$: $h a^{(g)}:=h(g)a^{(g)}$ for all $h \in (FG)^*$, $a^{(g)}\in L^{(g)}$
and $g\in G$. If $G$ is a finite group, then $(FG)^*$
  has a structure of a Hopf algebra and $L$ becomes an $(FG)^*$-module Lie algebra.

Let $m \geqslant 2$ be an integer and let $\zeta$ be a primitive $m$th root of unity
in a field $F$. (Such root exists in $F$ only if $\ch F \nmid m$.)
Consider the associative algebra $H_{m^2}(\zeta)$ with unity generated
by elements $c$ and $v$ satisfying the relations $c^m=1$, $v^m=0$, $vc=\zeta cv$.
Note that $(c^i v^k)_{0 \leqslant i, k \leqslant m-1}$ is a basis of $H_{m^2}(\zeta)$.
We introduce on $H_{m^2}(\zeta)$ a structure of a coalgebra by
  $\Delta(c)=c\otimes c$,
$\Delta(v) = c\otimes v + v\otimes 1$, $\varepsilon(c)=1$, $\varepsilon(v)=0$.
Then $H_{m^2}(\zeta)$ is a Hopf algebra with the antipode $S$ where $S(c)=c^{-1}$
and $S(v)=-c^{-1}v$. The algebra $H_{m^2}(\zeta)$ is called a \textit{Taft algebra}.

In the paper, each time we consider $H_{m^2}(\zeta)$-module algebras, we implicitly assume that the base 
field $F$ contains a primitive $m$th root of unity $\zeta$
and $\ch F \nmid m$.

\begin{remark}\label{RemarkTaftZmGrading}
Note that if $L$ is an $H_{m^2}(\zeta)$-module Lie algebra, then the cyclic group $\langle c \rangle_m \cong \mathbb Z_m := \mathbb Z / m\mathbb Z$
is acting on $L$ by automorphisms and $v$ is acting by a nilpotent skew-derivation. 
Every Lie algebra $L$ with a $\mathbb Z_m$-action by automorphisms is a $\mathbb Z_m$-graded Lie algebra: $$L = \bigoplus_{i=0}^{m-1} L^{(i)},\qquad L^{(i)} := \lbrace a \in L \mid ca = \zeta^i a\rbrace,\qquad
L^{(i)}L^{(k)}\subseteq L^{(i+k)}.$$ (When we consider $\mathbb Z_m$-gradings, all upper indices in parentheses are assumed to be modulo $m$.) Conversely, if $L = \bigoplus_{i=0}^{m-1} L^{(i)}$ is a $\mathbb Z_m$-graded Lie algebra,
then $\mathbb Z_m$ is acting on $L$ by automorphisms: $c a^{(i)} := \zeta^i a^{(i)}$
for all $a^{(i)} \in L^{(i)}$. Moreover, $L$ is \textit{$\mathbb Z_m$-simple} (i.e. $[L,L]\ne 0$ and $L$ has no non-trivial $\mathbb Z_m$-invariant ideals) if and only if $L$ is \textit{$\mathbb Z_m$-graded simple} (i.e. $[L,L]\ne 0$ and $L$ has no non-trivial ideals homogeneous in the $\mathbb Z_m$-grading). If $L$ is a $\mathbb Z_m$-graded Lie algebra, then its solvable and nilpotent
radicals are $\mathbb Z_m$-graded ideals since they are stable under the automorphism $c$.

In the article, each time we consider a $\mathbb Z_m$-grading
on an $H_{m^2}(\zeta)$-module algebra, we assume that this $\mathbb Z_m$-grading is induced by the $c$-action.
\end{remark}

$\mathbb Z_m$-graded modules over $\mathbb Z_m$-graded Lie algebras are defined in the natural way.
The analogs of the Weyl theorem on complete reducibility and the Jordan--H\"older Theorem
hold for them. (The proof of the first one can be found e.\,g. in~\cite[Lemma~3]{ASGordienko2} or~\cite[Theorem 9]{ASGordienko4}.
The second one is proved in the usual way considering graded submodules only.)

In the theorems below we use \textit{quantum binomial coefficients}: 
$$\binom{n}{0}_\zeta = \binom{n}{n}_\zeta := 1,\qquad
\binom{n}{k}_\zeta := \frac{n!_\zeta}{(n-k)!_\zeta\ k!_\zeta}$$ where $n!_\zeta := n_\zeta (n-1)_\zeta \cdot \dots \cdot 1_\zeta$ and
$n_\zeta := 1 + \zeta + \zeta^2 + \dots + \zeta^{n-1}$, $n\in\mathbb N$, $0_\zeta :=1$.

\section{Semisimple $H_{m^2}(\zeta)$-simple Lie algebras}\label{SectionClassTaftSLieSS}

In this section we classify semisimple $H_{m^2}(\zeta)$-simple Lie algebras which are non-simple
in the ordinary sense.

Let $B$ be a simple Lie algebra over a field $F$.
Suppose $F$ contains a primitive $m$th root of unity $\zeta$. 
Let $\alpha \in F$. Denote
$$L_\alpha(B) := \underbrace{B\oplus \ldots \oplus B}_{m} \text{ (direct sum of ideals)}.$$
Define the $c$- and $v$-action on $L_\alpha(B)$
by
\begin{equation}\label{EqLieTaftSimpleSSDefc}c(a_1, a_2, \ldots, a_{m-1}, a_m) = (a_m, a_1, a_2, \ldots, a_{m-1})\end{equation}
and \begin{equation}\label{EqLieTaftSimpleSSDefv}v(a_1, \ldots, a_m)=\alpha\left(a_1 - a_m, \zeta(a_2 - a_1), \ldots, \zeta^{m-1}(a_m - a_{m-1})\right)\end{equation} for all $a_1, \ldots, a_m \in B$.

By induction (the details can be found in~\cite[Lemma 3]{ASGordienko12}),
for arbitrary  $a_1, \ldots, a_m \in B$,
we get $$v^\ell (a_1, a_2, \ldots, a_m) =(b_1, b_2, \ldots, b_m)$$
where \begin{equation*} b_k = \alpha^\ell \zeta^{\ell(k-1)} \sum_{j=0}^\ell  (-1)^j \zeta^{-\frac{j(j-1)}{2}} \binom{\ell}{j}_{\zeta^{-1}}
a_{k-j} \end{equation*}
and $a_{-j} := a_{m-j} $ for $j \geqslant 0$.
Hence $v^m (a_1, a_2, \ldots, a_m) = 0$ for all $a_i \in B$.
An explicit check shows that~(\ref{EqLieTaftSimpleSSDefc}) and~(\ref{EqLieTaftSimpleSSDefv})
define the $H_{m^2}(\zeta)$-action on $L_\alpha(B)$ correctly and $L_\alpha(B)$ is $\mathbb Z_m$-graded simple and, therefore, $H_{m^2}(\zeta)$-simple.

Another description of $L_\alpha(B)$ for $\alpha \ne 0$
will be given in Theorem~\ref{TheoremTaftSimpleLieEquivDef} below.

\begin{theorem}\label{TheoremTaftSimpleSemiSimpleLieClassify}
Let $L$ be a finite dimensional semisimple $H_{m^2}(\zeta)$-simple Lie algebra over an algebraically closed field $F$ of characteristic $0$. Suppose $L$ is semisimple but not simple. Then
$L$ is a $\mathbb Z_m$-graded simple Lie algebra.
If $vL \ne 0$, then $L \cong L_\alpha(B)$ for some simple Lie algebra $B$ and some $\alpha \in F$.
\end{theorem}
\begin{proof}
First, $L=B_1 \oplus \ldots \oplus B_s$ (direct sum of ideals)
where $B_i$ are simple Lie algebras. Thus for every $1\leqslant i \leqslant s$
there exists $1\leqslant j(i) \leqslant s$ such that
$cB_i = B_{j(i)}$.
Moreover, $v[a,b]=[ca, vb]+[va, b] \in B_i \oplus B_{j(i)}$
for all $a,b \in B_i$. Since $[B_i, B_i] = B_i$, we get $vB_i \subseteq  B_i \oplus B_{j(i)}$.
In particular, the ideal $\sum_{k=0}^{m-1} c^k B_1$ is invariant under both $c$ and $v$.
Since $L$ is $H_{m^2}(\zeta)$-simple, we get $\sum_{k=0}^{m-1} c^k B_1 = L$
and, obviously, $L$ is $\mathbb Z_m$-simple and $\mathbb Z_m$-graded simple.
Without loss of generality, we may assume that $B_i = c^{i-1} B_1$.

Let $\pi_i \colon B \twoheadrightarrow B_i$ be the natural projection.
Define $\rho_i \colon B_i \to B_i$ and $\theta_i \colon B_i \to cB_i$
by $\rho_i(a)=\pi_i(va)$ and $\theta_i(a)=\pi_{i+1}(va)$
for all $a\in B_i$. (In the proof of the theorem all lower indices are assumed to be modulo $s$, e.g. $\pi_{s+1} := \pi_1$.)
Then $$\rho_i[a,b]=\pi_i(v[a,b])=\pi_i([ca,vb]+[va,b])=[\rho_i(a),b].$$
Analogously, $$\theta_i[a,b]=\pi_{i+1}(v[a,b])=\pi_{i+1}([ca,vb]+[va,b])=[ca, \theta_i(b)]$$
for all $a,b\in B_i$.
In particular, both $\rho_i$ and $\theta_i$ are homomorphisms of $B_i$-modules.

Since $B_i$ are simple Lie algebras, $B_i$ are irreducible $B_i$-modules.
Hence by the Schur lemma we have $\rho_i = \alpha_i \id_{B_i}$
and $\theta_i = \beta_i \left(c\bigr|_{B_i}\right)$ for some $\alpha_i, \beta_i \in F$.
Since $vc=\zeta cv$, we have $$\alpha_{i+1} ca =\rho_{i+1}(ca)=\zeta\pi_{i+1}(c(v(a)))=
\zeta c(\rho_i(a))=\zeta \alpha_i ca$$
and $$\beta_{i+1} c^2 a =\theta_{i+1}(ca)=\zeta\pi_{i+2}(c(v(a)))=
\zeta c(\theta_i(a))=\zeta \beta_i c^2 a$$ for all $1\leqslant i \leqslant s$
and $a\in B_i$.
 Hence $\alpha_i = \zeta^{i-1} \alpha_1$
and $\beta_i = \zeta^{i-1} \beta_1$ for all $1\leqslant i \leqslant s$.
Moreover, if at least one of $\alpha_1$ and $\beta_1$ is nonzero, we get $\zeta^s = 1$
and $s=m$.

Note that for all $1\leqslant i \leqslant s$, $a \in B_i$, and $b \in B_{i+1}$
we have $$0=v[a,b]=[ca,vb]+[va, b]=[ca,\rho_{i+1}(b)]+[\theta_i(a),b]=(\alpha_{i+1}+\beta_i) [ca, b].$$
Since $[B_{i+1}, B_{i+1}]=B_{i+1}$,
we obtain $\beta_i = -\alpha_{i+1}$ for all $1\leqslant i \leqslant s$.

If $\alpha_1 = 0$, then $vL=0$ and the theorem is proved.
Suppose $\alpha_1 \ne 0$. Then $s=m$.
Since $B_i = c^{i-1} B_1$ and $va=\rho_i(a)+\theta_i(a)$ for all $a\in B_i$, we may identify $B_i$ and assume that $L=\underbrace{B\oplus \ldots \oplus B}_{m}$ (direct sum of ideals)
for the simple Lie algebra $B:=B_1$,
where (\ref{EqLieTaftSimpleSSDefc}) and~(\ref{EqLieTaftSimpleSSDefv}) hold for $\alpha := \alpha_1$.
\end{proof}
\begin{remark}
If $vL=0$, then the proof of Theorem~\ref{TheoremTaftSimpleSemiSimpleLieClassify} 
shows that there exists $s\in\mathbb N$, $s \mid m$, and a simple Lie algebra $B$ with an action
of the cyclic group of order $\frac{m}{s}$ with a generator $d$ such that
$$L \cong \underbrace{B\oplus \ldots \oplus B}_{s} \text{ (direct sum of ideals)},$$
\begin{equation*}c(a_1, a_2, \ldots, a_{m-1}, a_m) = (d a_m, a_1, a_2, \ldots, a_{m-1})\end{equation*}
and \begin{equation*}v(a_1, \ldots, a_m)=0\end{equation*} for all $a_1, \ldots, a_m \in B$.
\end{remark}

\medskip

In Theorem~\ref{TheoremTaftSimpleSemiSimpleLieIso} below we give
necessary and sufficient conditions for $L_{\alpha_1}(B_1) \cong L_{\alpha_2}(B_2)$ as $H_{m^2}(\zeta)$-module Lie algebras.

\begin{theorem}\label{TheoremTaftSimpleSemiSimpleLieIso}
Let $B_1, B_2$ be simple Lie algebras over a field $F$, $\alpha_1, \alpha_2 \in F$.
Let $\zeta$ be a primitive $m$th root of unity.
Suppose $\theta \colon L_{\alpha_1}(B_1) \mathrel{\widetilde{\to}} L_{\alpha_2}(B_2)$ is an isomorphism
of Lie algebras and $H_{m^2}(\zeta)$-modules.
Then there exists $0\leqslant k \leqslant m-1$ and an isomorphism
of Lie algebras $\psi \colon B_1 \mathrel{\widetilde{\to}} B_2$ such that
\begin{equation}\label{EqLieTaftSimpleSSDefTheta}\theta(b_1, \ldots, b_m)=(\psi(b_{k+1}), \ldots, \psi(b_m), \psi(b_1), \ldots, \psi(b_k))\end{equation}
for all $b_i \in B_1$.
Moreover, $\alpha_2 = \zeta^k\alpha_1$.
Conversely, if $B_1 \cong B_2$ as ordinary Lie algebras and $\alpha_2 = \zeta^k\alpha_1$ for some $k\in\mathbb Z$, then $L_{\alpha_1}(B_1) \cong L_{\alpha_2}(B_2)$ as $H_{m^2}(\zeta)$-module Lie algebras.
\end{theorem}
\begin{proof}
Note that each minimal ideal of $L_{\alpha_2}(B_2)$ coincides
with one of the copies of $B_2$. Thus there exists $0\leqslant k \leqslant m-1$ such that $$\theta(B_1, 0, \ldots, 0)=
(\underbrace{0,\ldots, 0}_{m-k}, B_2, 0, \ldots, 0).$$
Denote the induced isomorphism $B_1 \mathrel{\widetilde{\to}} B_2$ by $\psi$. Then
$$\theta(b, 0, \ldots, 0)=
(\underbrace{0,\ldots, 0}_{m-k}, \psi(b), 0, \ldots, 0)$$
for all $b\in B$.
Now~(\ref{EqLieTaftSimpleSSDefc})
together with $\theta(ca)=c\theta(a)$ for all $a\in L_{\alpha_1}(B_1)$
implies~(\ref{EqLieTaftSimpleSSDefTheta}).
Using~(\ref{EqLieTaftSimpleSSDefv}) 
and $\theta(va)=v\theta(a)$ for all $a\in L_{\alpha_1}(B_1)$,
we get $\alpha_2 = \zeta^k\alpha_1$.
The converse is now evident.
\end{proof}
\begin{remark}
In particular, if $\alpha \ne 0$, then all automorphisms of $L_{\alpha}(B)$ as an $H_{m^2}(\zeta)$-module Lie algebra are induced by automorphisms of $B$ as an ordinary Lie algebra, and the corresponding automorphisms groups $\Aut_{H_{m^2}(\zeta)}(L_{\alpha}(B))$ and $\Aut(B)$ can be identified.
If $\alpha = 0$, then $\Aut_{H_{m^2}(\zeta)}(L_{\alpha}(B)) \cong \Aut(B) \times \mathbb Z_m$.
\end{remark}

\section{Lie algebras $L(B, \gamma)$ and $H_{m^2}(\zeta)$-actions on simple Lie algebras}\label{SectionClassTaftSLieSimple}

The next step in the classification of finite dimensional $H_{m^2}(\zeta)$-simple Lie algebras
is the study of $H_{m^2}(\zeta)$-actions on simple Lie algebras.
In fact, we will prove that finite dimensional simple Lie algebras $L$ endowed with
$H_{m^2}(\zeta)$-action have $va=0$ for all $a\in L$. (See Theorem~\ref{TheoremTaftSimpleSimpleLieClassify} below.)
In order to do this, we introduce $H_{m^2}(\zeta)$-simple Lie algebras $L(B, \gamma)$.

\begin{theorem}\label{TheoremTaftSimpleLiePresent} Let $B$ be a simple Lie algebra over a field $F$
and let $\gamma\in F$ be some element. Suppose $F$ contains some primitive $m$th root of unity $\zeta$. Define vector spaces $L^{(i)}$, $1\leqslant i \leqslant m-1$, $F$-linearly isomorphic to $L^{(0)} := B$. Let $\varphi \colon L^{(i-1)} \mathrel{\widetilde\to} L^{(i)}$, $1 \leqslant i \leqslant m-1$, be the corresponding $F$-linear bijections, which we denote by the same letter. Let $\varphi(L^{(m-1)}):=0$.
Consider the $H_{m^2}(\zeta)$-module $L(B,\gamma) :=\bigoplus_{i=0}^{m-1} L^{(i)}$ (direct sum of subspaces) where $v\varphi(a)=a$ for all $a \in L^{(i)}$,
$0\leqslant i \leqslant m-2$, $vB=0$,
and $c a^{(i)}=\zeta^i a^{(i)}$,  $a^{(i)} \in L^{(i)}$.
  Define the commutator on $L(B,\gamma)$ by 
  \begin{equation}\label{EqMultTaftSimpleLiePresent}
  [\varphi^k(a),\varphi^\ell(b)]:=\left\lbrace
\begin{array}{rrr}
  \binom{k+\ell}{k}_\zeta\ \varphi^{k+\ell}[a,b]  & \text{if} & k+\ell < m,\\
  \gamma\frac{(k+\ell-m)!_\zeta}{k!_\zeta \ell!_\zeta}\ \varphi^{k+\ell-m}[a,b] 
  & \text{if} & k+\ell \geqslant m
  \end{array}\right.
  \end{equation}
  for all $a, b\in B$ and $0 \leqslant k,\ell < m$.
    Then $L(B,\gamma)$ is an $H_{m^2}(\zeta)$-simple Lie algebra.
\end{theorem}
\begin{proof}
An explicit verification shows that the formulas indeed define on $L(B,\gamma)$ a structure of an $H_{m^2}(\zeta)$-module Lie algebra. Here we check only that $v[u,w]=[v_{(1)}u, v_{(2)}w]$
for all $u,w\in L(B,\gamma)$.

Let $0\leqslant k,\ell < m$ and $a,b\in B$.
If $k=\ell=0$, then $v[a,b]=0=[ca,vb]+[va,b]$.

If $k=0$, $\ell > 0$, then
  $$v[a, \varphi^\ell(b)]=v \varphi^\ell[a,b]= \varphi^{\ell-1}[a,b]=
 [a,\varphi^{\ell-1}(b)]
 =  [ca, v\varphi^\ell(b)]+[va, \varphi^\ell(b)].$$

If $k > 0$, $\ell=0$, 
 then $$v[\varphi^k(a), b]=v \varphi^k[a,b]= \varphi^{k-1}[a,b]=
 [\varphi^{k-1}(a), b]
 =  [c\varphi^k(a), vb]+[v\varphi^k(a), b].$$

If $k,\ell > 0$, $k+\ell < m$, then
 \begin{equation*}\begin{split}v[\varphi^k(a),\varphi^\ell(b)] = \binom{k+\ell}{k}_\zeta \varphi^{k+\ell-1}[a,b]
 = \left(\zeta^k\binom{k+\ell-1}{k}_\zeta+\binom{k+\ell-1}{k-1}_\zeta\right) \varphi^{k+\ell-1}[a,b]\\= [c\varphi^k(a),\varphi^{\ell-1}(b)] + [\varphi^{k-1}(a),\varphi^\ell(b)]= [c\varphi^k(a),v\varphi^\ell(b)] + [v\varphi^k(a),\varphi^\ell(b)]\end{split}\end{equation*}
since \begin{equation}\begin{split}\label{EquationMainQuantumBinomial}
\zeta^k \binom{k+\ell-1}{k}_\zeta + \binom{k+\ell-1}{k-1}_\zeta
= \frac{\zeta^k (k+\ell-1)!_\zeta}{k!_\zeta (\ell-1)!_\zeta}+
\frac{(k+\ell-1)!_\zeta}{(k-1)!_\zeta \ell!_\zeta}\\=(\zeta^k \ell_\zeta + k_\zeta)\frac{(k+\ell-1)!_\zeta}{k!_\zeta \ell!_\zeta}=(k+\ell)_\zeta \frac{(k+\ell-1)!_\zeta}{k!_\zeta \ell!_\zeta}
= \frac{(k+\ell)!_\zeta}{k!_\zeta \ell!_\zeta}=\binom{k+\ell}{k}_\zeta.\end{split}\end{equation}

If $k+\ell \geqslant m$, then $k,\ell > 0$.
If $k+\ell =m$, then $(k+\ell)_\zeta = m_\zeta = 0$
and
 \begin{equation*}\begin{split}v[\varphi^k(a),\varphi^\ell(b)] = \frac{\gamma}{k!_\zeta \ell!_\zeta}\ v[a,b]= 0 = \binom{k+\ell}{k}_\zeta \varphi^{k+\ell-1}[a,b]
 \\= \left(\zeta^k\binom{k+\ell-1}{k}_\zeta+\binom{k+\ell-1}{k-1}_\zeta\right) \varphi^{k+\ell-1}[a,b]= [c\varphi^k(a),\varphi^{\ell-1}(b)] + [\varphi^{k-1}(a),\varphi^\ell(b)]\\= [c\varphi^k(a),v\varphi^\ell(b)] + [v\varphi^k(a),\varphi^\ell(b)].\end{split}\end{equation*}

If $k+\ell >m$, then $(k+\ell)_\zeta = (k+\ell-m)_\zeta + \zeta^{k+\ell-m} m_\zeta = (k+\ell-m)_\zeta$
and
 \begin{equation*}\begin{split}v[\varphi^k(a),\varphi^\ell(b)] = \gamma\frac{(k+\ell-m)!_\zeta}{k!_\zeta \ell!_\zeta}\ \varphi^{k+\ell-m-1}[a,b] 
 \\=
\gamma(k+\ell)_\zeta\frac{(k+\ell-m-1)!_\zeta}{k!_\zeta \ell!_\zeta}\ \varphi^{k+\ell-m-1}[a,b] 
 \\= 
  \gamma\left(\zeta^k \frac{(k+\ell-m-1)!_\zeta}{k!_\zeta (\ell-1)!_\zeta}+\frac{(k+\ell-m-1)!_\zeta}{(k-1)!_\zeta \ell!_\zeta}\right) \varphi^{k+\ell-m-1}[a,b]\\= [c\varphi^k(a),\varphi^{\ell-1}(b)] + [\varphi^{k-1}(a),\varphi^\ell(b)]\\= [c\varphi^k(a),v\varphi^\ell(b)] + [v\varphi^k(a),\varphi^\ell(b)].\end{split}\end{equation*}

We have considered all possible variants for $0\leqslant k,\ell < m$.
Hence $v[u,w]=[v_{(1)}u, v_{(2)}w]$
for all $u,w\in L(B,\gamma)$.

Suppose that $I$ is an $H_{m^2}(\zeta)$-invariant ideal of $L$. Then $v^m I = 0$.
Let $t \in \mathbb Z_+$ such that $v^t I \ne 0$, $v^{t+1} I = 0$. Then $0 \ne v^t I \subseteq I \cap \ker v$. However, $\ker v = B$ is a simple Lie algebra. Thus $I \cap \ker v =\ker v$ and $\ker v \subseteq I$. Since $$[\ker v, L^{(i)}]=[B,\varphi^i(B)]=\varphi^i[B,B]=\varphi^i(B)=L^{(i)}\text{ for all }1\leqslant i \leqslant m-1,$$
we get $I = L(B,\gamma)$. Therefore, $L(B,\gamma)$ is an  $H_{m^2}(\zeta)$-simple Lie algebra.
\end{proof}
\begin{remark}
Lie algebras $L(B,0)$ are not semisimple. The solvable radical of $L(B,0)$ coincides with
the nilpotent radical and equals $\bigoplus_{i=1}^{m-1} L^{(i)}$.
\end{remark}

In Theorem~\ref{TheoremTaftSimpleLieEquivDef} below
we prove that if the field $F$ is algebraically closed and $\gamma \ne 0$,
then $L(B, \gamma)$ is isomorphic to one of the non-simple $\mathbb Z_m$-graded simple $H_{m^2}(\zeta)$-module Lie algebras $L_\alpha(B)$ defined in Section~\ref{SectionClassTaftSLieSS}.

\begin{theorem}\label{TheoremTaftSimpleLieEquivDef}
Let $B$ be a simple Lie algebra over a field $F$.
 Suppose $F$ contains some primitive $m$th root of unity $\zeta$. Let $\alpha \in F$, $\alpha \ne 0$.
Then $L(B, \frac{1}{\alpha^m (1-\zeta)^m}) \cong L_\alpha(B)$ 
as $H_{m^2}(\zeta)$-module Lie algebras.
 \end{theorem}
\begin{proof}
Note that $$L_\alpha(B)^{(k)}=\left\lbrace\left(b, \zeta^{-k} b, \zeta^{-2k} b, \ldots, \zeta^{-(m-1)k}b\right)
\mathrel{\bigl|} b \in B\right\rbrace$$ for $0\leqslant k \leqslant m-1$. In particular, $L_\alpha(B)^{(0)}\cong B$.
Define $$\varphi(b, \zeta^{-k} b, \zeta^{-2k} b, \ldots, \zeta^{-(m-1)k}b)
:= \frac{1}{\alpha(1-\zeta^{k+1})}\left(b, \zeta^{-(k+1)} b, \zeta^{-2(k+1)} b, \ldots, \zeta^{-(m-1)(k+1)}b\right)$$
for all $b\in B$ and $0\leqslant k < m-1$.
Then  \begin{equation*}\begin{split}\varphi^k(b,b,\ldots, b) = \frac{1}{\alpha^k(1-\zeta)(1-\zeta^2)\ldots (1-\zeta^k)}
(b, \zeta^{-k} b, \zeta^{-2k} b, \ldots, \zeta^{-(m-1)k}b)\\=\frac{1}{\alpha^k(1-\zeta)^k k!_\zeta}
(b, \zeta^{-k} b, \zeta^{-2k} b, \ldots, \zeta^{-(m-1)k}b).\end{split}\end{equation*}
Note that
$L_\alpha(B)^{(k)}=\varphi^k(L_\alpha(B)^{(0)})$ and $v\varphi(a)=a$ for all $a\in L_\alpha(B)^{(k)}$, $0\leqslant k < m-1$. 
Moreover, $[\varphi^k(a),\varphi^\ell(b)]$ can be calculated using~(\ref{EqMultTaftSimpleLiePresent})
for $\gamma=\frac{1}{\alpha^m (1-\zeta)^m}$ for all $a,b \in L_\alpha(B)^{(0)}$ and $0\leqslant k,\ell < m$.
Hence $L_\alpha(B) \cong L(B, \frac{1}{\alpha^m (1-\zeta)^m})$ as $H_{m^2}(\zeta)$-module Lie algebras.
\end{proof}

Now we prove several lemmas on $H_{m^2}(\zeta)$-module
Lie algebras.

\begin{lemma}\label{LemmaLieTaftSimpleSGradVGrad}
Let $L$ be an $H_{m^2}(\zeta)$-module Lie algebra over a field $F$. 
Then 
\begin{equation}\label{EqLieTaftSimpleSGradMultLR}(\zeta^k-1)[a^{(k)},v b^{(\ell)}] = (\zeta^\ell-1) [v a^{(k)}, b^{(\ell)}],\end{equation}
\begin{equation}\label{EqLieTaftSimpleSGradMultProd}(\zeta^\ell-1)v[a^{(k)},b^{(\ell)}]=(\zeta^{k+\ell}-1) [a^{(k)},v b^{(\ell)}]\end{equation}
for all $a^{(k)}\in L^{(k)}$, $b^{(\ell)}\in L^{(\ell)}$
in the natural $\mathbb Z_m$-grading induced by the $c$-action.
Moreover, if $L$ is a $\mathbb Z_m$-graded simple Lie algebra with respect to this grading,
$vL^{(0)}=0$.
\end{lemma}
\begin{proof}
Note that $$v[a^{(k)},b^{(\ell)}]=[ca^{(k)}, vb^{(\ell)}]+[va^{(k)}, b^{(\ell)}]=\zeta^k [a^{(k)},v b^{(\ell)}]
+[va^{(k)}, b^{(\ell)}]$$ for all $a^{(k)}\in L^{(k)}$
and $b^{(\ell)} \in L^{(\ell)}$.
At the same time 
\begin{equation*}\begin{split}v[a^{(k)},b^{(\ell)}]=-v[b^{(\ell)},a^{(k)}]=-[cb^{(\ell)}, va^{(k)}]-[vb^{(\ell)}, a^{(k)}]\\ = -\zeta^\ell [b^{(\ell)},v a^{(k)}]
-[vb^{(\ell)}, a^{(k)}]=[a^{(k)}, vb^{(\ell)}]+\zeta^\ell [v a^{(k)}, b^{(\ell)}].\end{split}\end{equation*}
Hence we obtain~(\ref{EqLieTaftSimpleSGradMultLR})
and~(\ref{EqLieTaftSimpleSGradMultProd})
and if $\ell \ne 0$, we get
$$ v[a^{(k)},b^{(\ell)}]=\frac{\zeta^{k+\ell}-1}{\zeta^\ell-1} [a^{(k)},v b^{(\ell)}].
$$
In particular, $v[L^{(k)}, L^{(m-k)}]=0$ for all $1\leqslant k \leqslant m-1$.

Suppose $L$ is a $\mathbb Z_m$-graded simple Lie algebra.
If $L= L^{(0)}$, then $c$ is acting trivially and $vc=\zeta cv$
implies $vL=0$.
Therefore, we may assume that $L\ne L^{(0)}$. Let $a \in L^{(k)}$, $k\ne 0$, $a\ne 0$.
Since $L$ is $\mathbb Z_m$-graded simple and $a$ is homogeneous, $a$ generates $L$ as an ideal.
Thus $L^{(0)}$ is an $F$-linear span of elements
$[a, a_1, \ldots, a_n]$, where $a_i \in L^{(k_i)}$, $0\leqslant k_i \leqslant m-1$,
$m \mid (k+k_1+\ldots+k_n)$, $n\in \mathbb N$.
(Here we use long commutators $[x_1, \ldots, x_n]
:= [[\ldots [[x_1, x_2], x_3], \ldots], x_n]$.)
If $k_n \ne 0$,
then $[a, a_1, \ldots, a_n]=[[a, a_1, \ldots, a_{n-1}], a_n]$
implies $[a, a_1, \ldots, a_n]\in [L^{(m-k_n)}, L^{(k_n)}]$
and $v[a, a_1, \ldots, a_n]=0$.
If $k_n = 0$, we apply the Jacobi identity
and rewrite $[a, a_1, \ldots, a_n]$ as a sum of
$[[a,a_n], a_1, \ldots, a_{n-1}]$, $[a, [a_1,a_n], \ldots, a_{n-1}]$, and
 $[a, a_1, \ldots, [a_{n-1},a_n]]$. If $a_{n-1} \in L^{(0)}$, we continue this procedure.
 Finally, we get the situation where the last components in long commutators
 belong to $L^{(k_i)}$, $k_i \ne 0$. Applying the arguments used above,
 we get $v[a, a_1, \ldots, a_n]=0$. 
 Hence $v L^{(0)}=0$.
\end{proof}

\begin{lemma}\label{LemmaLieTaftSimpleSGradTrois}
Let $L$ be an $H_{m^2}(\zeta)$-module Lie algebra over a field $F$ of characteristic $\ch F \nmid m$, $\ch F \ne 2$. Let $a^{(\ell)} \in L^{(\ell)}$, $b^{(k)}\in L^{(k)}$, $u^{(m-k)}\in L^{(m-k)}$ for some $1\leqslant k, \ell \leqslant m-1$.
Then $$v[a^{(\ell)},[b^{(k)}, u^{(m-k)}]] = \frac{\zeta^\ell-1}{\zeta^{m-k}-1}[a^{(\ell)},[b^{(k)}, vu^{(m-k)}]].$$
\end{lemma}
\begin{proof} By~(\ref{EqLieTaftSimpleSGradMultLR}), (\ref{EqLieTaftSimpleSGradMultProd}) and the Jacobi identity,
\begin{equation*}\begin{split}[a^{(\ell)},[b^{(k)}, vu^{(m-k)}]]=-\frac{\zeta^{m-k}-1}{\zeta^{k}-1}[a^{(\ell)},[u^{(m-k)}, vb^{(k)}]]
\\ =-\frac{\zeta^{m-k}-1}{\zeta^{k}-1}([[a^{(\ell)},u^{(m-k)}], vb^{(k)}]+ [u^{(m-k)}, [a^{(\ell)},vb^{(k)}]]) \\ =-\frac{\zeta^{m-k}-1}{\zeta^{\ell}-1}v[a^{(\ell)},u^{(m-k)}, b^{(k)}]
+\frac{\zeta^{m-k}-1}{\zeta^{\ell}-1}[u^{(m-k)}, [b^{(k)},va^{(\ell)}]]
\\=-\frac{\zeta^{m-k}-1}{\zeta^{\ell}-1}v[a^{(\ell)},u^{(m-k)}, b^{(k)}]
+\frac{\zeta^{m-k}-1}{\zeta^{\ell}-1}([[u^{(m-k)}, b^{(k)}],va^{(\ell)}]
+ [b^{(k)},[u^{(m-k)}, va^{(\ell)}]])\\=\frac{\zeta^{m-k}-1}{\zeta^{\ell}-1}v\left(-[a^{(\ell)},u^{(m-k)}, b^{(k)}]+[u^{(m-k)}, b^{(k)},a^{(\ell)}]\right)-[b^{(k)},[a^{(\ell)}, vu^{(m-k)}]]\\=
\frac{\zeta^{m-k}-1}{\zeta^{\ell}-1}v\left(-[a^{(\ell)},u^{(m-k)}, b^{(k)}]+[u^{(m-k)}, b^{(k)},a^{(\ell)}]\right)\\-[[b^{(k)},a^{(\ell)}], vu^{(m-k)}]-[a^{(\ell)}, [b^{(k)},vu^{(m-k)}]]
\\=\frac{\zeta^{m-k}-1}{\zeta^{\ell}-1}v\left(-[a^{(\ell)},u^{(m-k)}, b^{(k)}]+[u^{(m-k)}, b^{(k)},a^{(\ell)}]-[b^{(k)},a^{(\ell)}, u^{(m-k)}]\right)\\-[a^{(\ell)}, [b^{(k)},vu^{(m-k)}]]
\\= 2\frac{\zeta^{m-k}-1}{\zeta^{\ell}-1} v[a^{(\ell)}, [b^{(k)}, u^{(m-k)}]]-[a^{(\ell)}, [b^{(k)},vu^{(m-k)}]]
\end{split}\end{equation*} and the lemma follows.
\end{proof}

\begin{lemma}\label{LemmaLieTaftSimpleSGradVMany}
Let $L$ be an $H_{m^2}(\zeta)$-module Lie algebra over a field $F$ of characteristic $\ch F \nmid m$, $\ch F \ne 2$. Suppose $L$ is a $\mathbb Z_m$-graded simple Lie algebra.
Let $s \geqslant 2$ and $0\leqslant k_i \leqslant m-1$ for $1\leqslant i \leqslant s$,
$k_s > 0$. Let $a_i^{(k_i)} \in L_i^{(k_i)}$ for $1\leqslant i \leqslant s$.
Then 
\begin{equation}\label{EqLieTaftSimpleSGradVMany}
v[a_1^{(k_1)}, [a_2^{(k_2)}, \ldots, [a_{s-1}^{(k_{s-1})}, a_s^{(k_s)}]\ldots]
 = \frac{\zeta^{\sum_{i=1}^s k_i}-1}{\zeta^{k_s}-1}[a_1^{(k_1)}, [a_2^{(k_2)}, \ldots, [a_{s-1}^{(k_{s-1})}, v a_s^{(k_s)}]\ldots].\end{equation}
\end{lemma}
\begin{proof}
We prove the assertion by induction on $s$. The base $s=2$ is a consequence of~(\ref{EqLieTaftSimpleSGradMultProd}).
Suppose $s > 2$.

 If $m \nmid \sum_{i=2}^s k_i$,
then by~(\ref{EqLieTaftSimpleSGradMultProd}) and the induction assumption,  \begin{equation*}\begin{split}v[a_1^{(k_1)}, [a_2^{(k_2)}, \ldots, [a_{s-1}^{(k_{s-1})}, a_s^{(k_s)}]\ldots]
 =  \frac{\zeta^{\sum_{i=1}^s k_i}-1}{\zeta^{\sum_{i=2}^s k_i}-1}[a_1^{(k_1)}, v[a_2^{(k_2)}, \ldots, [a_{s-1}^{(k_{s-1})}, a_s^{(k_s)}]\ldots]
  \\ =\frac{\zeta^{\sum_{i=1}^s k_i}-1}{\zeta^{k_s}-1}[a_1^{(k_1)}, [a_2^{(k_2)}, \ldots, [a_{s-1}^{(k_{s-1})}, v a_s^{(k_s)}]\ldots]\end{split}\end{equation*}
and the lemma follows.

Suppose $k_i = 0$ for some $1\leqslant i < s$.
Then by the Jacobi identity (here the symbol $\widehat{a_i^{(k_i)}}$ means that the element $a_i^{(k_i)}$ is omitted), \begin{equation*}\begin{split}a:=v[a_1^{(k_1)}, [a_2^{(k_2)}, \ldots, [a_{s-1}^{(k_{s-1})}, a_s^{(k_s)}]\ldots]  \\
 = \sum_{j=i+1}^{s-1} v[a_1^{(k_1)}, [a_2^{(k_2)}, \ldots, [\widehat{a_i^{(k_i)}}, \ldots, [a_{j-1}^{(k_{j-1})}, [[a_i^{(k_i)}, a_j^{(k_j)}],
 [a_{j+1}^{(k_{j+1})}, \ldots [a_{s-1}^{(k_{s-1})}, a_s^{(k_s)}]\ldots]
 \\+ v[a_1^{(k_1)}, [a_2^{(k_2)}, \ldots, [\widehat{a_i^{(k_i)}}, \ldots, [a_{s-1}^{(k_{s-1})}, [a_i^{(k_i)}, a_s^{(k_s)}]]\ldots].\end{split}\end{equation*}

Since $k_i=0$, each $[a_i^{(k_i)}, a_j^{(k_j)}]$ is again of degree $k_j$ and we treat it as a single element.
Applying the induction assumption for $s-1$, we get
\begin{equation*}\begin{split}a 
 =  \frac{\zeta^{\sum_{i=1}^s k_i}-1}{\zeta^{k_s}-1} \\ \cdot\left(\sum_{j=i+1}^{s-1} [a_1^{(k_1)}, [a_2^{(k_2)}, \ldots, [\widehat{a_i^{(k_i)}}, \ldots, [a_{j-1}^{(k_{j-1})}, [[a_i^{(k_i)}, a_j^{(k_j)}],
 [a_{j+1}^{(k_{j+1})}, \ldots [a_{s-1}^{(k_{s-1})}, v a_s^{(k_s)}]\ldots]
 \right. \\ \left.+ [a_1^{(k_1)}, [a_2^{(k_2)}, \ldots, [\widehat{a_i^{(k_i)}}, \ldots, [a_{s-1}^{(k_{s-1})}, v[a_i^{(k_i)}, a_s^{(k_s)}]]\ldots]\right).\end{split}\end{equation*}
 By~(\ref{EqLieTaftSimpleSGradMultProd}) and the Jacobi identity, we get
  \begin{equation*}\begin{split}a 
 = \frac{\zeta^{\sum_{i=1}^s k_i}-1}{\zeta^{k_s}-1} \\ \cdot\left(\sum_{j=i+1}^{s-1} [a_1^{(k_1)}, [a_2^{(k_2)}, \ldots, [\widehat{a_i^{(k_i)}}, \ldots, [a_{j-1}^{(k_{j-1})}, [[a_i^{(k_i)}, a_j^{(k_j)}],
 [a_{j+1}^{(k_{j+1})}, \ldots [a_{s-1}^{(k_{s-1})}, v a_s^{(k_s)}]\ldots]
  \right. \\ \left.+[a_1^{(k_1)}, [a_2^{(k_2)}, \ldots, [\widehat{a_i^{(k_i)}}, \ldots, [a_{s-1}^{(k_{s-1})}, [a_i^{(k_i)}, v a_s^{(k_s)}]]\ldots]\right)
  \\= \frac{\zeta^{\sum_{i=1}^s k_i}-1}{\zeta^{k_s}-1}[a_1^{(k_1)}, [a_2^{(k_2)}, \ldots, [a_{s-1}^{(k_{s-1})}, v a_s^{(k_s)}]\ldots]\end{split}\end{equation*}
 and the lemma again follows.

The only case we have not considered yet is when all $1\leqslant k_i \leqslant m-1$
and $m \mid \sum_{i=2}^s k_i$. But then $m \mid \left(k_2 + \sum_{i=3}^s k_i\right)$
and we can apply Lemma~\ref{LemmaLieTaftSimpleSGradTrois}:
\begin{equation*}\begin{split} v[a_1^{(k_1)}, [a_2^{(k_2)}, [a_3^{(k_3)} \ldots, [a_{s-1}^{(k_{s-1})}, a_s^{(k_s)}]\ldots] \\ = \frac{\zeta^{\sum_{i=1}^s k_i}-1}{\zeta^{\sum_{i=3}^s k_i}-1}[a_1^{(k_1)}, [a_2^{(k_2)}, v\bigl[a_3^{(k_3)} \ldots, [a_{s-1}^{(k_{s-1})}, a_s^{(k_s)}]\ldots\bigr]]] 
 \end{split}\end{equation*}
and the lemma follows by the induction assumption for $s-2$.\end{proof}

\begin{lemma}\label{LemmaLieTaftSimpleSGradVKer}
Let $L$ be an $H_{m^2}(\zeta)$-module Lie algebra over a field $F$ of characteristic $\ch F \nmid m$, $\ch F \ne 2$. Suppose $L$ is a $\mathbb Z_m$-graded simple Lie algebra, $vL\ne 0$.
Then $\ker v = L^{(0)}$.
\end{lemma}
\begin{proof}
By Lemma~\ref{LemmaLieTaftSimpleSGradVGrad}, $vL^{(0)}=0$. Since $vc=\zeta cv$, $\ker v$ is a $\mathbb Z_m$-graded subspace. Suppose $\ker v \supsetneqq L^{(0)}$. Then there exists an element $a^{(k)} \subseteq L^{(k)}$,
$1\leqslant k \leqslant m-1$, $a^{(k)}\ne 0$, $va^{(k)}=0$.
Since $L$ is $\mathbb Z_m$-graded simple and $a^{(k)}$ is homogeneous, we have $L=\sum_{n \geqslant 0}[\underbrace{L, [L,\ldots [L,}_{n} F a^{(k)}]\ldots]$. By Lemma~\ref{LemmaLieTaftSimpleSGradVMany},
we get $vL=0$.
\end{proof}

\begin{lemma}\label{LemmaLieTaftSimpleSGradVIm}
Let $L$ be an $H_{m^2}(\zeta)$-module Lie algebra over a field $F$ of characteristic $\ch F \nmid m$. Suppose $L$ is a $\mathbb Z_m$-graded simple Lie algebra, $vL\ne 0$.
Then  $v L^{(k)}=L^{(k-1)}$ for all $1\leqslant k \leqslant m-1$.
\end{lemma}
\begin{proof}
First, we claim that $L = vL + \sum_{k\ne 1}^{m-1} [L^{(k)}, vL^{(m-k)}]$.
Note that $$I=vL + \sum_{k= 1}^{m-1} [L^{(k)}, vL^{(m-k)}]$$ is a $\mathbb Z_m$-graded subspace.
We claim that $I$ is an ideal too. 

By~(\ref{EqLieTaftSimpleSGradMultProd}),
$[L^{(\ell)}, vL^{(k)}] \subseteq vL$ for all $0\leqslant k,\ell < m$
such that $m \nmid (k+\ell)$.
Hence $[L, vL] \subseteq I$.

Now we show that $\left[L, \sum_{k= 1}^{m-1} [L^{(k)}, vL^{(m-k)}]\right] \subseteq I$.

First, the Jacobi identity implies \begin{equation}\label{EqLieTaftSimpleSGradJacobi}[a^{(\ell)}, [b^{(k)}, vu^{(m-k)}]]
= [[a^{(\ell)}, b^{(k)}], vu^{(m-k)}]+ [ b^{(k)}, [a^{(\ell)},vu^{(m-k)}]]\end{equation}
for all $0\leqslant k,\ell < m$ and $a^{(\ell)} \in L^{(\ell)}$, $b^{(k)} \in L^{(k)}$, $u^{(m-k)}\in L^{(m-k)}$.

If $\ell = 0$, then by~(\ref{EqLieTaftSimpleSGradMultProd})
and~(\ref{EqLieTaftSimpleSGradJacobi}),
$$[a^{(\ell)}, [b^{(k)}, vu^{(m-k)}]]
= [[a^{(\ell)}, b^{(k)}], vu^{(m-k)}]+ [ b^{(k)}, v[a^{(\ell)},u^{(m-k)}]] \in [L^{(k)}, vL^{(m-k)}].$$

If $\ell\ne 0$ and $k \ne \ell$, then by~(\ref{EqLieTaftSimpleSGradMultProd})
and~(\ref{EqLieTaftSimpleSGradJacobi}),
\begin{equation*}[a^{(\ell)}, [b^{(k)}, vu^{(m-k)}]]= \frac{\zeta^{m-k}-1}{\zeta^{\ell}-1}v\left([[a^{(\ell)}, b^{(k)}], u^{(m-k)}]+[ b^{(k)}, [a^{(\ell)},u^{(m-k)}]]\right) \in vL.\end{equation*}

Suppose $\ell \ne 0$ and $k=\ell$.
Below we show that $[L^{(k)}, [L^{(k)},vL^{(m-k)}]] \subseteq vL$.

If $m \ne 2k$, then $m-k \ne k$. By~(\ref{EqLieTaftSimpleSGradMultLR}), (\ref{EqLieTaftSimpleSGradMultProd}) and the Jacobi identity,
\begin{equation*}\begin{split}[L^{(k)}, [L^{(k)},vL^{(m-k)}]] \subseteq [L^{(k)}, [L^{(m-k)}, vL^{(k)}]]
\\ \subseteq [[L^{(k)}, L^{(m-k)}], vL^{(k)}]
+ [L^{(m-k)}, [L^{(k)}, vL^{(k)}]] \subseteq vL. \end{split}\end{equation*}

If $m=2k$, then $\ch F \ne 2$ and the inclusion $[L^{(k)}, [L^{(k)},vL^{(k)}]] \subseteq vL$ is a consequence of Lemma~\ref{LemmaLieTaftSimpleSGradTrois}.
 
 Thus $I$ is indeed a $\mathbb Z_m$-graded ideal and $L = vL + \sum_{k\ne 0} [L^{(k)}, vL^{(m-k)}]$ since
$L$ is $\mathbb Z_m$-graded simple.

Since $vc = \zeta cv$, we have $vL^{(k)} \subseteq L^{(k-1)}$, and $\sum_{k\ne 0} [L^{(k)}, vL^{(m-k)}]\subseteq L^{(m-1)}$. Thus $\bigoplus_{k=0}^{m-2} L^{(k)} \subseteq vL$.
Since by Lemma~\ref{LemmaLieTaftSimpleSGradVGrad} we have $vL^{(0)}=0$,
this implies
$vL \cap L^{(m-1)}=0$ and $vL = \bigoplus_{k=0}^{m-2} L^{(k)}$. In particular, $v L^{(k)}=L^{(k-1)}$ for all $1\leqslant k \leqslant m-1$.
\end{proof}

\begin{lemma}\label{LemmaLieTaftSimpleSGradPhiMult}
Let $L$ be an $H_{m^2}(\zeta)$-module Lie algebra over a field $F$ of characteristic $\ch F \nmid m$, $\ch F \ne 2$. Suppose $L$ is a $\mathbb Z_m$-graded simple Lie algebra, $vL\ne 0$.
Define the maps $\varphi \colon L^{(k)} \to L^{(k+1)}$ (we denote them by the same letter)
 by $\varphi(va)=a$ for $a \in L^{(k+1)}$, $0\leqslant k \leqslant m-2$.
Define $\lbrace a, b \rbrace := (m-1)!_\zeta [\varphi(a), \varphi^{m-1}(b)]$.
Then 
  \begin{equation}\label{EqLieTaftSimpleSGradPhiMult}
  [\varphi^k(a),\varphi^\ell(b)]:=\left\lbrace
\begin{array}{rrr}
  \binom{k+\ell}{k}_\zeta\ \varphi^{k+\ell}[a,b]  & \text{if} & k+\ell < m,\\
  \frac{(k+\ell-m)!_\zeta}{k!_\zeta \ell!_\zeta}\ \varphi^{k+\ell-m}\lbrace a,b \rbrace 
  & \text{if} & k+\ell \geqslant m
  \end{array}\right.
  \end{equation}
  for all $a, b\in L^{(0)}$ and $0 \leqslant k,\ell < m$.
\end{lemma}
\begin{proof}
By Lemmas~\ref{LemmaLieTaftSimpleSGradVKer} and~\ref{LemmaLieTaftSimpleSGradVIm},
$\varphi$ is well-defined. Moreover,
 $v\varphi(a)=a$ for all $a\in L^{(i)}$, $0\leqslant i \leqslant m-2$.

If $k=\ell=0$, then~(\ref{EqLieTaftSimpleSGradPhiMult}) is trivial.
By Lemma~\ref{LemmaLieTaftSimpleSGradVGrad}, $v[\varphi^k(a),b]=[\varphi^{k-1}(a),b]$
for $1\leqslant k \leqslant m-1$ and $a,b\in L^{(0)}$.
Hence $[\varphi^k(a),b]=\varphi^k[a,b]$
and we get~(\ref{EqLieTaftSimpleSGradPhiMult})
in the case when at least one of $k,\ell$ is zero.

The case of arbitrary $k,\ell > 0$, $k+\ell < m$, is done by induction
using~(\ref{EquationMainQuantumBinomial}):
\begin{equation*}\begin{split}
[\varphi^k(a),\varphi^\ell(b)]= \varphi \left(v[\varphi^k(a),\varphi^\ell(b)]\right)
=\varphi([c\varphi^k(a), \varphi^{\ell-1}(b)]+[\varphi^{k-1}(a),\varphi^\ell(b)])\\
=\varphi\left([\zeta^k\varphi^k(a), \varphi^{\ell-1}(b)]+[\varphi^{k-1}(a),\varphi^\ell(b)]\right)\\
=\varphi\left(\zeta^k\binom{k+\ell-1}{k}_\zeta\ \varphi^{k+\ell-1}[a,b]+\binom{k+\ell-1}{k-1}_\zeta\ \varphi^{k+\ell-1}[a,b] \right)\\
=\left(\zeta^k \binom{k+\ell-1}{k}_\zeta + \binom{k+\ell-1}{k-1}_\zeta\right) \varphi^{k+\ell}[a,b]
\\ =\binom{k+\ell}{k}_\zeta\ \varphi^{k+\ell}[a,b].
\end{split}\end{equation*}

Suppose $k+\ell = m$.
We prove the assertion by induction on $k$.
 If $k=1$ and $\ell = m-1$, then (\ref{EqLieTaftSimpleSGradPhiMult}) follows from the definition
of $\lbrace , \rbrace$.
If $k>1$,
then $\ell < m-1$ and by~(\ref{EqLieTaftSimpleSGradMultLR})
and the induction assumption for $k-1$, we have
\begin{equation*}\begin{split}[\varphi^k(a),\varphi^\ell(b)]=
[\varphi^k(a),v\varphi^{\ell+1}(b)]
=\frac{\zeta^{\ell+1}-1}{\zeta^k-1}
[v\varphi^k(a), \varphi^{\ell+1}(b)]=\frac{\zeta^{\ell+1}-1}{\zeta^k-1}[\varphi^{k-1}(a), \varphi^{\ell+1}(b)]
\\=\frac{\zeta^{\ell+1}-1}{(\zeta^k-1)(k-1)!_\zeta (\ell+1)!_\zeta } \lbrace a,b\rbrace =
\frac{(\ell+1)_\zeta}{k_\zeta(k-1)!_\zeta (\ell+1)!_\zeta } \lbrace a,b\rbrace=\frac{1}{k!_\zeta \ell!_\zeta } \lbrace a,b\rbrace.\end{split}\end{equation*}

If $k+\ell > m$, then we again use induction on $(k+\ell)$:
\begin{equation*}\begin{split}
[\varphi^k(a),\varphi^\ell(b)]= \varphi \left(v[\varphi^k(a),\varphi^\ell(b)]\right)
=\varphi([c\varphi^k(a), \varphi^{\ell-1}(b)]+[\varphi^{k-1}(a),\varphi^\ell(b)])\\
=\varphi\left([\zeta^k\varphi^k(a), \varphi^{\ell-1}(b)]+[\varphi^{k-1}(a),\varphi^\ell(b)]\right)\\
=\varphi\left(\zeta^k\frac{(k+\ell-m-1)!_\zeta}{k!_\zeta
(\ell-1)!_\zeta}\ \varphi^{k+\ell-m-1}\lbrace a,b\rbrace+\frac{(k+\ell-m-1)!_\zeta}{(k-1)!_\zeta
\ell!_\zeta}\ \varphi^{k+\ell-m-1}\lbrace a,b\rbrace \right)\\=
\left(\zeta^k \frac{(k+\ell-m-1)!_\zeta}{k!_\zeta
(\ell-1)!_\zeta} + \frac{(k+\ell-m-1)!_\zeta}{(k-1)!_\zeta
\ell!_\zeta}\right) \varphi^{k+\ell-m}\lbrace a,b\rbrace
\\ =\frac{(k+\ell)_\zeta (k+\ell-m-1)!_\zeta}{k!_\zeta
\ell!_\zeta}\ \varphi^{k+\ell-m}\lbrace a,b\rbrace=\frac{(k+\ell-m)!_\zeta}{k!_\zeta
\ell!_\zeta}\ \varphi^{k+\ell-m}\lbrace a,b\rbrace
\end{split}\end{equation*}
since $(k+\ell)_\zeta=(k+\ell-m)_\zeta$ for $m<k+\ell < 2m$.
\end{proof}

\begin{lemma}\label{LemmaLieTaftSimpleSGradIsoToLBGamma}
Let $L$ be a finite dimensional $H_{m^2}(\zeta)$-module Lie algebra over an algebraically closed field $F$ of characteristic $0$. Suppose $L$ is a $\mathbb Z_m$-graded simple Lie algebra, $vL\ne 0$.
Then $L^{(0)}$ is a simple Lie algebra and there exist $\gamma \in F$, $\gamma \ne 0$,
such that $\lbrace a, b \rbrace = \gamma [a,b]$ for all $a,b \in L^{(0)}$.
In other words, $L \cong L(L^{(0)}, \gamma)$ as an $H_{m^2}(\zeta)$-module Lie algebra.
\end{lemma}
\begin{proof}
Note that by the Jacobi identity we have $$[\varphi^{m-1}(a), \varphi(b), u]
+ [\varphi(b), u, \varphi^{m-1}(a)]
+ [u, \varphi^{m-1}(a), \varphi(b)]=0$$
for all $a,b,u \in L^{(0)}$.
Together with Lemma~\ref{LemmaLieTaftSimpleSGradPhiMult}
this implies 
\begin{equation}\label{EqLieTaftSimpleSGradIsoToLBGamma1}
[\lbrace a, b \rbrace, u]+\lbrace [b,u], a \rbrace
+ \lbrace [u,a], b \rbrace = 0.
\end{equation}

Again, by the Jacobi identity we have $$[\varphi^{m-1}(a), \varphi^{m-1}(b), \varphi(u)]
+ [\varphi^{m-1}(b), \varphi(u), \varphi^{m-1}(a)]
+ [\varphi(u), \varphi^{m-1}(a), \varphi^{m-1}(b)]=0$$
for all $a,b,u \in L^{(0)}$.
Together with Lemma~\ref{LemmaLieTaftSimpleSGradPhiMult}
this implies 
\begin{equation}\label{EqLieTaftSimpleSGradIsoToLBGamma2}
[\lbrace a, b \rbrace, u]+[\lbrace b,u \rbrace, a]
+ [\lbrace u, a \rbrace, b] = 0.
\end{equation}

Now~(\ref{EqLieTaftSimpleSGradIsoToLBGamma1})
implies
$$[\lbrace b, u \rbrace, a]+\lbrace [u,a], b \rbrace
+ \lbrace [a,b], u \rbrace = 0, $$
$$[\lbrace u, a \rbrace, b]+\lbrace [a,b], u \rbrace
+ \lbrace [b,u], a \rbrace = 0.$$
Summing this up with~(\ref{EqLieTaftSimpleSGradIsoToLBGamma1})
and using~(\ref{EqLieTaftSimpleSGradIsoToLBGamma2}) and $\ch F \ne 2$,
we get
$$\lbrace [a,b], u \rbrace +\lbrace [b,u], a \rbrace
+ \lbrace [u,a], b \rbrace =0$$ for all $a,b,u \in L^{(0)}$.
Together with~(\ref{EqLieTaftSimpleSGradIsoToLBGamma1})
this implies \begin{equation}\label{EqLieTaftSimpleSGradIsoToLBGamma3}\lbrace [a,b], u \rbrace = [\lbrace a,b\rbrace, u].\end{equation}
By Lemma~\ref{LemmaLieTaftSimpleSGradPhiMult},
$\lbrace a, b \rbrace = -\lbrace b,a \rbrace$.
Together with~(\ref{EqLieTaftSimpleSGradIsoToLBGamma2}) and~(\ref{EqLieTaftSimpleSGradIsoToLBGamma3})
this implies 
$$\lbrace [a,b], u \rbrace = [\lbrace a, u\rbrace, b] + [a, \lbrace b, u \rbrace] $$ for all $a,b,u \in L^{(0)}$.
In other words, $\lbrace \cdot, u \rbrace$ is a derivation for
all $u \in L^{(0)}$.

Now we show that $L^{(0)}$ is a simple Lie algebra.
Suppose first $I \ne 0$ is an ideal of $L^{(0)}$
such that $\lbrace I, u \rbrace \subseteq I$ for all $u \in L^{(0)}$.
Then by Lemma~\ref{LemmaLieTaftSimpleSGradPhiMult},
$\bigoplus_{k=0}^{m-1} \varphi^k(I)$ is a nonzero $\mathbb Z_m$-graded ideal of $L$.
Hence $L = \bigoplus_{k=0}^{m-1} \varphi^k(I)$ and $I=L^{(0)}$.
By~\cite[Chapter~III, Section~6, Theorem~7]{JacobsonLie}, the solvable
radical of~$L^{(0)}$  is invariant under all derivations.
Hence $L^{(0)}$ is semisimple.
If $L^{(0)}$ is non-simple, then 
$L^{(0)}=I_1 \oplus I_2$ for some nonzero ideals $I_1$ and $I_2$.
Let $\delta$ be a derivation of $L^{(0)}$.
Then $\delta(I_i)=\delta[I_i, I_i]\subseteq[\delta(I_i), I_i]+[I_i, \delta(I_i)]\subseteq I_i$.
In other words, $I_i$, $i=1,2$, are invariant under all derivations and $L^{(0)}=I_1 = I_2$,
i.e. we get a contradiction. Therefore, $L^{(0)}$ is a simple Lie algebra.

Since $L^{(0)}$ is simple, all derivations of $L^{(0)}$ are inner.
Hence there exists an $F$-linear map $\psi \colon L^{(0)} \to L^{(0)}$
such that $\lbrace a,b \rbrace = [a, \psi(b)]$
for all $a,b\in L^{(0)}$. Now~(\ref{EqLieTaftSimpleSGradIsoToLBGamma3})
implies $$[\psi[a,b], u] = \lbrace [a,b], u \rbrace
= [\lbrace a,b\rbrace, u] = [[a,\psi(b)], u]$$
for all $a,b,u \in L^{(0)}.$
Since $L^{(0)}$ has zero center, we have $\psi[a,b] = [a,\psi(b)]$
for all $a,b \in L^{(0)}$. In other words,
$\psi \colon L^{(0)} \to L^{(0)}$ is a homomorphism
of $L^{(0)}$-modules. Since $L^{(0)}$ is an irreducible
 $L^{(0)}$-module, $\psi$ is a scalar map and $\lbrace a,b \rbrace = \gamma[a, b]$
 for some $\gamma \in F$. By Lemma~\ref{LemmaLieTaftSimpleSGradPhiMult},
 $L \cong L(L^{(0)}, \gamma)$ as an $H_{m^2}(\zeta)$-module Lie algebra.
 Since $L$ is $\mathbb Z_m$-graded simple and, therefore, semisimple, we have $\gamma \ne 0$.
\end{proof}

In Theorem~\ref{TheoremTaftSimpleSimpleLieClassify} below we show that
each finite dimensional $H_{m^2}(\zeta)$-module Lie algebra
simple in the ordinary sense is just a $\mathbb Z_m$-graded Lie algebra
with the trivial $v$-action.

\begin{theorem}\label{TheoremTaftSimpleSimpleLieClassify}
Let $L$ be a finite dimensional $H_{m^2}(\zeta)$-module Lie algebra over an algebraically closed field $F$ of characteristic $0$. Suppose $L$ is simple in the ordinary sense. Then $vL = 0$.
\end{theorem}
\begin{proof} 
Suppose $vL\ne 0$.
Then by Lemma~\ref{LemmaLieTaftSimpleSGradIsoToLBGamma} we have the isomorphism
$L \cong L(L^{(0)}, \gamma)$ of $H_{m^2}(\zeta)$-module Lie algebras
for some $\gamma \ne 0$.
By Theorem~\ref{TheoremTaftSimpleLieEquivDef},  
$L(L^{(0)}, \gamma) \cong L_\alpha(L^{(0)})$
where $\alpha =  \frac{(1-\zeta)^{-1}}{\sqrt[m]{\gamma}}$. However $L_\alpha(L^{(0)})$
is non-simple as an ordinary Lie algebra and we get a contradiction.
Thus $vL=0$. \end{proof}

\section{Non-semisimple $H_{m^2}(\zeta)$-simple Lie algebras}\label{SectionClassTaftSLieNSS}

In this section we show that all non-semisimple $H_{m^2}(\zeta)$-simple Lie algebras
are isomorphic to Lie algebras from Theorem~\ref{TheoremTaftSimpleLiePresent} with $\gamma = 0$.

\begin{theorem}\label{TheoremTaftSimpleNonSemiSimpleLieClassify}
Suppose $L$ is a finite dimensional $H_{m^2}(\zeta)$-simple Lie algebra over an algebraically closed field $F$ of characteristic $0$ and the solvable radical of $L$ is nonzero.
Then $L$ is isomorphic as an $H_{m^2}(\zeta)$-module Lie algebra to the Lie algebra $L(B, 0)$ for some finite dimensional simple Lie algebra $B$.
\end{theorem}

In order to prove Theorem~\ref{TheoremTaftSimpleNonSemiSimpleLieClassify}, we need several auxiliary lemmas.

Let $M_1,M_2$ be two $\mathbb Z_m$-graded modules over a $\mathbb Z_m$-graded Lie algebra $L$. We say that a $F$-linear bijection $\varphi \colon M_1 \mathrel{\widetilde\to} M_2$ is a \textit{$c$-isomorphism}
of $M_1$ and $M_2$ if there exists $r\in\mathbb Z$ such that $c\varphi(b) = \zeta^{-r} \varphi(cb)$, $\varphi(ab)=a\varphi(b)$ for all $b\in M_1$, $a\in L$.

Recall that for any finite-dimensional Lie algebra $L$ over a field of characteristic $0$ we have
$[L,R]\subseteq N$ (see e.g.~\cite[Proposition 2.1.7]{GotoGrosshans}) where $R$, $N$ are, respectively, the solvable and the nilpotent radical. Hence if $N=0$, then we have $R \subseteq Z(L)\subseteq N=0$
where $Z(L)$ is the center of $L$. 
Recall also that if $L$ is an $H_{m^2}(\zeta)$-module Lie algebra, then $R$ and $N$ are $\mathbb Z_m$-graded ideals since $R$ and $N$ are invariant under all automorphisms of $L$ and, in particular, under the $c$-action. 

\begin{lemma}\label{LemmaTaftSimpleNonSemiSimpleClassifyLieSumDirect}
Suppose $L$ is a finite dimensional $H_{m^2}(\zeta)$-simple Lie algebra over a field $F$ and its solvable radical $R \ne 0$.
Let $N$ be the nilpotent radical of $L$,
 $N^\ell = 0$, $N^{\ell-1} \ne 0$. Choose a minimal $\mathbb Z_m$-graded $L$-ideal $\tilde N \subseteq N^{\ell-1}$.
Then for any $k$ the subspace $N_k := \sum_{i=0}^{i=k} v^i \tilde N$ is a $\mathbb Z_m$-graded ideal of $L$ and $L = \bigoplus_{i=0}^t v^i \tilde N$ (direct sum of $\mathbb Z_m$-graded subspaces) for some $1 \leqslant t \leqslant m-1$.
Moreover, $N_k/N_{k-1}$, $0 \leqslant k \leqslant t$, are irreducible $\mathbb Z_m$-graded $L$-modules $c$-isomorphic to each other. (Here $N_{-1} := 0$.)
\end{lemma}
\begin{proof}
Since for any $a \in \tilde N$ and $b\in L$ the element $[v^k a, b]=v[v^{k-1} a, b]-[cv^{k-1} a, vb]$ can be presented as an $F$-linear combination
of elements $v^i[c^{k-i} a, v^{k-i}b]$, each $N_k := \sum_{i=0}^{i=k} v^i \tilde N$ is a $\mathbb Z_m$-graded ideal of $L$.

 Recall that $v^m =0$. Thus $N_{m-1}$ is an $H_{m^2}(\zeta)$-invariant ideal of $L$.
Hence  $L=N_{m-1}$.

Let $\varphi_k \colon N_k/N_{k-1} \twoheadrightarrow N_{k+1}/N_k$, where $0 \leqslant k \leqslant m-2$, be the map
defined by $$\varphi_k (b + N_{k-1}) = vb + N_k.$$ 
Denote $\bar b:= b + N_{k-1}$.
Then $c\varphi_k (\bar b) = \zeta^{-1}\varphi_k (c\bar b)$,
\begin{equation*}\begin{split}
 \varphi_k(a \bar b) = v[a,b]+N_k = -v[b,a] + N_k= -[cb,va]-[vb,a]+N_k  \\=
-[vb,a]+N_k=[a,vb]+N_k = a\varphi_k (\bar b)
\text{ for all }a\in L,\ b \in N_k.\end{split}
\end{equation*}
Note that $\tilde N = N_0/N_{-1}$ is an irreducible $\mathbb Z_m$-graded $L$-module.
Therefore, $N_{k+1}/N_k = \varphi_k(N_k/N_{k-1})$ is an irreducible $\mathbb Z_m$-graded $L$-module or zero for any $0 \leqslant k < m-1$. Thus if $L = N_t$, $L \ne N_{t-1}$,
then $\dim N_t = (t+1)\dim \tilde N$ and $L = \bigoplus_{i=0}^t v^i \tilde N$ (direct sum of $\mathbb Z_m$-graded subspaces).
\end{proof}

\begin{lemma}\label{LemmaTaftSimpleNonSemiSimpleLieClassifyUnity}
Assume that we are under the conditions of Lemma~\ref{LemmaTaftSimpleNonSemiSimpleClassifyLieSumDirect}.
In addition, suppose that the field $F$ is algebraically closed  of characteristic $0$.
Then $R=N=N_{t-1}$, $L^{(i)} = v^{t-i} \tilde N$ for $0\leqslant i \leqslant t$, and $L^{(0)}\cong L/R$ is a simple Lie algebra. 
Moreover $\dim (N_k/N_{k-1})=\dim (L/R)$ for all $0 \leqslant k\leqslant t$.
In addition, $\ker v = L^{(0)}$.
\end{lemma}
\begin{proof}
First we notice that $[L,L]$ is an $H_{m^2}(\zeta)$-invariant ideal. Hence $L=[L,L]$ and
$L\ne R$.

By~\cite{Taft}, there exists a maximal $\mathbb Z_m$-graded semisimple Lie subalgebra $B\subseteq L$
such that $L = B \oplus R$ (direct sum of $\mathbb Z_m$-graded subspaces), $B \cong L/R$.
Note that $N$ annihilates all irreducible $\mathbb Z_m$-graded $L$-modules
that are factors of the adjoint representation of $L$. In addition, $[L, R] \subseteq N$ (see e.g.~\cite[Proposition 2.1.7]{GotoGrosshans}).
Hence $L/N$ is a reductive $\mathbb Z_m$-graded
Lie algebra and $\tilde N$ is an irreducible $\mathbb Z_m$-graded 
$L/N$-module. By~\cite[Lemma~6]{ASGordienko2},
we have $\tilde N = \bigoplus_{i=0}^s c^i M$
where $M\subseteq \tilde N$ is an $L$-submodule such that $R$ is acting on $M$ by scalar operators, $s \in\mathbb N$.
Since $\tilde N$  is an irreducible 
$\mathbb Z_m$-graded $L$-module, we may assume that $M$ is an irreducible $L$-module.
All $B$-submodules in $M$ are $L$-submodules
since  $R$ is acting on $M$ by scalar operators.
Hence $M$ is an irreducible $B$-module.
 
Since $\tilde N=N_0/N_{-1}$, all $N_k/N_{k-1}$ are $c$-isomorphic to each other, and $B$ is semisimple, the Lie algebra $L$ is a direct sum of irreducible $B$-submodules isomorphic to $c^j M$ where $j \in\mathbb Z$.
Note that the $B$-action on each $M$ and therefore on each $c^j M$
must be nonzero, since $B$ itself is a $B$-submodule of $L$
with a nonzero $B$-action.
On the other hand, there exists a $B$-submodule $Q \subseteq R$ such that $R=N\oplus Q$. Since $[L,R]\subseteq N$, we have $[B,Q]\subseteq N \cap Q = 0$,
i.e. $Q\subseteq L$ is a submodule with the zero $B$-action. Hence $Q=0$ and $R=N$. In particular, all
$N_k/N_{k-1}$ are irreducible $\mathbb Z_m$-graded $B$-modules
$c$-isomorphic to each other.
However, $B\subseteq L$ is a $\mathbb Z_m$-graded $B$-submodule.
If $B$ is not a $\mathbb Z_m$-graded simple Lie algebra, then $B$ is a direct sum
of $\mathbb Z_m$-graded simple Lie subalgebras (this follows e.g. from~\cite[Theorem~9]{ASGordienko4}), which are non-$c$-isomorphic as $B$-modules.
Hence $B$ must be a $\mathbb Z_m$-graded simple Lie algebra
and all $N_k/N_{k-1}$ are $c$-isomorphic to $B$ as $\mathbb Z_m$-graded modules.
Let $B \subseteq N_q$, $B \subsetneqq N_{q-1}$ for some $q \in\mathbb Z_+$.
If $q < t$, then $[B, N_t] \subseteq N_{t-1}$
and $N_t/N_{t-1}$ has the zero $B$-action. Since all $N_k/N_{k-1}$ are $c$-isomorphic,
we get a contradiction. Therefore, $B \cap N_{t-1} = 0$, $L=B \oplus N_{t-1}$ (direct sum of subspaces), and $B \cong L/N_{t-1}$.

We claim that $ca=a$ for all $a\in B$ and therefore $B$ is simple as an ordinary Lie algebra.
In Lemma~\ref{LemmaTaftSimpleNonSemiSimpleClassifyLieSumDirect}
we proved that $\varphi_k(a\bar b) = a\varphi_k(\bar b)$ 
for all $0 \leqslant k \leqslant m-1$ and $a \in L$, $b\in N_k$.
Analogously, one shows that $$\varphi_k(a\bar b) = v[a,b]+N_k = [ca, vb] + [va, b]+ N_k = 
[ca, vb] + N_k = (ca)\varphi_k(\bar b).$$
In other words, $((ca)-a)$ is acting as $0$ on all $N_k/N_{k-1}$
for every $a\in B$. In particular, $((ca)-a)$ belongs to the center of $B$.
Since $B$ is semisimple, we get $ca=a$ for all $a\in B$ and $B \subseteq L^{(0)}$ has a trivial grading.
Hence $B$ is simple as an ordinary Lie algebra.

Note that $B \cong L/N_{t-1} \cong v^t \tilde N$ as $\mathbb Z_m$-graded spaces.
Hence $v^t \tilde N \subseteq L^{(0)}$.  Using
$vc=\zeta cv$, we get $v^i \tilde N \subseteq L^{(t-i)}$.
Since $L=\bigoplus_{i=0}^{m-1} L^{(i)} = \bigoplus_{i=0}^t v^i \tilde N$, we obtain
$L^{(i)} = v^{t-i} \tilde N$ for $0\leqslant i \leqslant t$, $L^{(i)}=0$
for $t+1\leqslant i \leqslant m-1$.
In particular, $B=L^{(0)}=v^t \tilde N$.

Recall that each $N_j = \bigoplus_{i=0}^j v^i \tilde N=\bigoplus_{i=t-j}^t L^{(i)}$ is an ideal.
Hence for $0 \leqslant j \leqslant t$ and $0 \leqslant i \leqslant m-1$
we always have $0 \leqslant t-j+i < t-j+m$,
$$[L^{(i)}, L^{(t-j)}] \subseteq N_j \cap L^{(t-j+i)} \subseteq N_{j-i},$$
and $[L^{(i)}, N_k] \subseteq N_{k-i}$. (We assume that $N_k := 0$ for $k<0$.)
In particular, the ideal $N_{t-1}$ is nilpotent and $R=N=N_{t-1}$.

If $t=m-1$, then $vL^{(0)}=v^m \tilde N = 0$. Since $N_{t-1} \cap (\ker v) = 0$, we get $\ker v =L^{(0)}$.
If $t < m-1$, then $vL^{(0)} \subseteq L^{(m-1)} = 0$. Again, $\ker v =L^{(0)}$.
\end{proof}

\begin{lemma}\label{LemmaTaftSimpleNonSemiSimpleClassifyLieFormula}
Suppose we are under the assumptions of Lemma~\ref{LemmaTaftSimpleNonSemiSimpleLieClassifyUnity}.
Define the $F$-linear map $\varphi \colon L \to L$ by $\varphi(v^k a) = v^{k-1} a$
for all $a \in\tilde N$, $1 \leqslant k \leqslant t$, $\varphi(\tilde N) = 0$.
Then \begin{equation}\label{EqQuantumBinomPhiLie}
[\varphi^k(a),\varphi^\ell(b)]=\binom{k+\ell}{k}_\zeta\ \varphi^{k+\ell}[a,b] \text{ for all }a, b\in L^{(0)} \text{ and } 0 \leqslant k,\ell \leqslant t.
\end{equation}
\end{lemma}
\begin{proof}
Note that $\varphi(va)=a$ for all $a\in N_{t-1}$.
Thus for every $b\in L^{(0)}$, $a=vu$, $u\in N_{t-1}$
we have $$c\varphi(a)=c\varphi(vu)=cu=\varphi(vcu)=\zeta\varphi(cvu)=\zeta\varphi(ca),$$
$$\varphi[a,b]=\varphi[vu,b]=\varphi(v[u,b]-[cu, vb])=[u,b]=[\varphi(a),b].$$
Since  $\varphi(\tilde N) = 0$, $L = v N_{t-1} \oplus \tilde N$ (direct sum of $\mathbb Z_m$-graded subspaces),
and $\tilde N$ is an ideal,
we have $c\varphi(a) = \zeta \varphi(ca)$,  
$\varphi[a,b]=[\varphi(a),b]$  for all $a\in L$, $b\in L^{(0)}$.
This proves~(\ref{EqQuantumBinomPhiLie}) for $k=0$ or $\ell=0$.

The case of arbitrary $k,\ell \geqslant 1$ is done by induction
using
\begin{equation*}\begin{split}
[\varphi^k(a),\varphi^\ell(b)]= \varphi \left(v[\varphi^k(a),\varphi^\ell(b)]\right)
=\varphi([c\varphi^k(a), \varphi^{\ell-1}(b)]+[\varphi^{k-1}(a),\varphi^\ell(b)])\end{split}\end{equation*}
analogously to Lemma~\ref{LemmaLieTaftSimpleSGradPhiMult}.
\end{proof}

\begin{proof}[Proof of Theorem~\ref{TheoremTaftSimpleNonSemiSimpleLieClassify}.]
By Lemma~\ref{LemmaTaftSimpleNonSemiSimpleLieClassifyUnity}, 
$L^{(0)}$ is a simple Lie algebra.
Let $a,b \in L^{(0)}$ such that $[a,b]\ne 0$.
Then $\varphi^t [a,b]\ne 0$.

Note that $[\varphi^t(a), \varphi(b)]=\binom{t+1}{t}_\zeta \varphi^{t+1}[a,b] = 0$.
However \begin{equation*}\begin{split}0=v[\varphi^t(a), \varphi(b)]
=[v\varphi^t(a), \varphi(b)]+[c\varphi^t(a), v\varphi(b)]
\\= \left(\binom t{t-1}_\zeta+\zeta^t\right)\varphi^t[a,b]
=(t+1)_\zeta\ \varphi^t[a,b].\end{split}\end{equation*} Hence $(t+1)_\zeta=0$ and $m=t+1$.
Now~(\ref{EqQuantumBinomPhiLie}) and Lemma~\ref{LemmaTaftSimpleNonSemiSimpleLieClassifyUnity} imply the theorem.
\end{proof}

\begin{remark}
Since the maximal semisimple Lie subalgebra $\ker v$ is uniquely determined, any two such $H_{m^2}(\zeta)$-simple Lie algebras $L$ are isomorphic as $H_{m^2}(\zeta)$-module Lie algebras if and only if their Lie subalgebras $\ker v$ are isomorphic as ordinary algebras. Moreover, all automorphisms of $L$ as an $H_{m^2}(\zeta)$-module Lie algebra are induced by the automorphisms of $\ker v$ as a Lie algebra.
Indeed, let $\psi \colon L \mathrel{\widetilde\to} L$ be an automorphism of $L$ as an $H_{m^2}(\zeta)$-module Lie algebra. Since $\tilde N = N^{m-1}$, we have $\psi(\tilde N) = \tilde N$
and $\psi(v^k \tilde N)=v^k \tilde N$ for
all $0\leqslant k < m$. Now $v^k \psi(\varphi^k(a))=\psi(a)$
for all $a\in \ker v$ implies $\psi(\varphi^k(a))=\varphi^k(\psi(a))$ and $\psi$ is uniquely determined by its restriction
on $\ker v$.
\end{remark}

\section{Polynomial $H$-identities}\label{SectionHPI}

In Section~\ref{SectionHPI-expTaftSLie} we prove that if $L$ is a finite dimensional $H_{m^2}(\zeta)$-simple
$H_{m^2}(\zeta)$-module Lie algebra over an algebraically closed field of characteristic $0$, then $\PIexp^{H_{m^2}(\zeta)}(L)=\dim L$. In particular, the $H_{m^2}(\zeta)$-PI-exponent of $L$ is integer
and the analog of Amitsur's conjecture holds for polynomial $H_{m^2}(\zeta)$-identities of $L$.

  Let $F$ be a field and let $F \lbrace X \rbrace$ be the absolutely free nonassociative algebra
   on the set $X := \lbrace x_1, x_2, x_3, \ldots \rbrace$.
  Then $F \lbrace X \rbrace = \bigoplus_{n=1}^\infty F \lbrace X \rbrace^{(n)}$
  where $F \lbrace X \rbrace^{(n)}$ is the $F$-linear span of all monomials of total degree $n$.
   Let $H$ be a Hopf algebra over a field $F$. Consider the algebra $$F \lbrace X | H\rbrace
   :=  \bigoplus_{n=1}^\infty H^{{}\otimes n} \otimes F \lbrace X \rbrace^{(n)}$$
   with the multiplication $(u_1 \otimes w_1)(u_2 \otimes w_2):=(u_1 \otimes u_2) \otimes w_1w_2$
   for all $u_1 \in  H^{{}\otimes j}$, $u_2 \in  H^{{}\otimes k}$,
   $w_1 \in F \lbrace X \rbrace^{(j)}$, $w_2 \in F \lbrace X \rbrace^{(k)}$.
We use the notation $$x^{h_1}_{i_1}
x^{h_2}_{i_2}\ldots x^{h_n}_{i_n} := (h_1 \otimes h_2 \otimes \ldots \otimes h_n) \otimes x_{i_1}
x_{i_2}\ldots x_{i_n}$$ (the arrangements of brackets on $x_{i_j}$ and on $x^{h_j}_{i_j}$
are the same). Here $h_1 \otimes h_2 \otimes \ldots \otimes h_n \in H^{{}\otimes n}$,
$x_{i_1} x_{i_2}\ldots x_{i_n} \in F \lbrace X \rbrace^{(n)}$. 

Note that if $(\gamma_\beta)_{\beta \in \Lambda}$ is a basis in $H$, 
then $F \lbrace X | H\rbrace$ is isomorphic to the absolutely free nonassociative algebra over $F$ with free formal  generators $x_i^{\gamma_\beta}$, $\beta \in \Lambda$, $i \in \mathbb N$.
 
    Define on $F \lbrace X | H\rbrace$ the structure of a left $H$-module
   by $$h\,(x^{h_1}_{i_1}
x^{h_2}_{i_2}\ldots x^{h_n}_{i_n})=x^{h_{(1)}h_1}_{i_1}
x^{h_{(2)}h_2}_{i_2}\ldots x^{h_{(n)}h_n}_{i_n},$$
where $h_{(1)}\otimes h_{(2)} \otimes \ldots \otimes h_{(n)}$
is the image of $h$ under the comultiplication $\Delta$
applied $(n-1)$ times, $h\in H$. Then $F \lbrace X | H\rbrace$ is \textit{the absolutely free $H$-module nonassociative algebra} on $X$, i.e. for each map $\psi \colon X \to A$ where $A$ is an $H$-module algebra,
there exists the unique homomorphism $\bar\psi \colon 
F \lbrace X | H\rbrace \to A$ of algebras and $H$-modules, such that $\bar\psi\bigl|_X=\psi$.
Here we identify $X$ with the set $\lbrace x^{1_H}_j \mid j \in \mathbb N\rbrace \subset F \lbrace X | H\rbrace$.

Consider the $H$-invariant ideal $I$ in $F\lbrace X | H \rbrace$
generated by the set \begin{equation}\label{EqSetOfHGen}
\bigl\lbrace u(vw)+v(wu)+w(uv) \mid u,v,w \in  F\lbrace X | H \rbrace\bigr\rbrace \cup\bigl\lbrace u^2 \mid u \in  F\lbrace X | H \rbrace\bigr\rbrace.
\end{equation}
 Then $L(X | H) := F\lbrace X | H \rbrace/I$
is \textit{the free $H$-module Lie algebra}
on $X$, i.e. for any $H$-module Lie algebra $L$ 
and a map $\psi \colon X \to L$, there exists the unique homomorphism $\bar\psi \colon L(X | H) \to L$
of Lie algebras and $H$-modules such that $\bar\psi\bigl|_X =\psi$. 
 We refer to the elements of $L(X | H)$ as \textit{Lie $H$-polynomials}.


\begin{remark} If $H$ is cocommutative and $\ch F \ne 2$, then $L(X | H)$ is the ordinary
free Lie algebra with free generators $x_i^{\gamma_\beta}$, $\beta \in \Lambda$, $i \in \mathbb N$
where   $(\gamma_\beta)_{\beta \in \Lambda}$ is a basis in $H$, since the ordinary ideal of 
$F\lbrace X | H \rbrace$ generated by~(\ref{EqSetOfHGen})
is already $H$-invariant.
However, if $h_{(1)} \otimes h_{(2)} \ne h_{(2)} \otimes h_{(1)}$ for some $h \in H$,
we still have $$[x^{h_{(1)}}_i, x^{h_{(2)}}_j]=h[x_i, x_j]=-h[x_j, x_i]=-[x^{h_{(1)}}_j, x^{h_{(2)}}_i]
= [x^{h_{(2)}}_i, x^{h_{(1)}}_j]$$ in $L(X | H)$ for all $i,j \in\mathbb N$,
i.e. in the case $h_{(1)} \otimes h_{(2)} \ne h_{(2)} \otimes h_{(1)}$ the Lie algebra $L(X | H)$ is not free as an ordinary Lie algebra.
\end{remark}

Let $L$ be an $H$-module Lie algebra for
some Hopf algebra $H$ over a field $F$.
 An $H$-polynomial
 $f \in L ( X | H )$
 is a \textit{$H$-identity} of $L$ if $\psi(f)=0$
for all homomorphisms $\psi \colon L(X|H) \to L$
of Lie algebras and $H$-modules. In other words, $f(x_1, x_2, \ldots, x_n)$
 is a polynomial $H$-identity of $L$
if and only if $f(a_1, a_2, \ldots, a_n)=0$ for any $a_i \in L$.
 In this case we write $f \equiv 0$.
The set $\Id^H(L)$ of all polynomial $H$-identities
of $L$ is an $H$-invariant ideal of $L(X|H)$.

\begin{example} Note that if $m=2$ and $\zeta=-1$, then $(1-c)L(\mathfrak{sl}_2(F),0)=L(\mathfrak{sl}_2(F),0)^{(1)}$
and the commutator of any two elements of $L(\mathfrak{sl}_2(F),0)^{(1)}$ is zero
by~\eqref{EqMultTaftSimpleLiePresent}. Hence
$$[x^{1-c},y^{1-c}]\in \Id^{H_4(-1)}(L(\mathfrak{sl}_2(F),0)).$$
\end{example}

Denote by $V^H_n$ the space of all multilinear Lie $H$-polynomials
in $x_1, \ldots, x_n$, $n\in\mathbb N$, i.e.
$$V^{H}_n = \langle [x^{h_1}_{\sigma(1)},
x^{h_2}_{\sigma(2)}, \ldots, x^{h_n}_{\sigma(n)}]
\mid h_i \in H, \sigma\in S_n \rangle_F \subset L( X | H ).$$
Then the number $c^H_n(L):=\dim\left(\frac{V^H_n}{V^H_n \cap \Id^H(L)}\right)$
is called the $n$th \textit{codimension of polynomial $H$-identities}
or the $n$th \textit{$H$-codimension} of $L$.

\begin{remark}
One can treat polynomial $H$-identities of $L$ as $H$-identities of a nonassociative
$H$-module algebra (i.e. use $F\lbrace X | H\rbrace$ instead of $L(X | H)$) and define their codimensions.
However those codimensions will coincide with $c^{H}_n(L)$ since the $n$th $H$-codimension
equals the dimension of the subspace in $\Hom_F(L^{{}\otimes n}; L)$ that consists of those $n$-linear functions that can be represented by $H$-polynomials.
\end{remark}

Recall that the limit $\PIexp^H(L):=\lim\limits_{n\to\infty}
 \sqrt[n]{c^H_n(L)}$ (if it exists) is called the \textit{$H$-PI-exponent} of $L$.

One of the main tools in the investigation of polynomial
identities is provided by the representation theory of symmetric groups.
 The symmetric group $S_n$  acts
 on the space $\frac {V^H_n}{V^H_{n}
  \cap \Id^H(L)}$
  by permuting the variables.
   If the base field $F$ is of characteristic $0$,
  then irreducible $FS_n$-modules are described by partitions
  $\lambda=(\lambda_1, \ldots, \lambda_s)\vdash n$ and their
  Young diagrams $D_\lambda$.
   The character $\chi^H_n(L)$ of the
  $FS_n$-module $\frac {V^H_n}{V^H_n
   \cap \Id^H(L)}$ is
   called the $n$th
  \textit{cocharacter} of polynomial $H$-identities of $L$.
  We can rewrite it as
  a sum $$\chi^H_n(L)=\sum_{\lambda \vdash n}
   m(L, H, \lambda)\chi(\lambda)$$ of
  irreducible characters $\chi(\lambda)$.
Let  $e_{T_{\lambda}}=a_{T_{\lambda}} b_{T_{\lambda}}$
and
$e^{*}_{T_{\lambda}}=b_{T_{\lambda}} a_{T_{\lambda}}$
where
$a_{T_{\lambda}} = \sum_{\pi \in R_{T_\lambda}} \pi$
and
$b_{T_{\lambda}} = \sum_{\sigma \in C_{T_\lambda}}
 (\sign \sigma) \sigma$,
be the Young symmetrizers corresponding to a Young tableau~$T_\lambda$.
Then $M(\lambda) = FS_n e_{T_\lambda} \cong FS_n e^{*}_{T_\lambda}$
is an irreducible $FS_n$-module corresponding to
 a partition~$\lambda \vdash n$.
  We refer the reader to~\cite{Bahturin, DrenKurs, ZaiGia}
   for an account
  of $S_n$-representations and their applications to polynomial
  identities.

 \section{Exponent of $H_{m^2}(\zeta)$-identities of $H_{m^2}(\zeta)$-simple Lie algebras}\label{SectionHPI-expTaftSLie}

  In this section we prove the existence of the $H_{m^2}(\zeta)$-PI-exponent for $H_{m^2}(\zeta)$-simple Lie algebras:
 
   \begin{theorem}\label{TheoremTaftSimpleLieHPIexpExists} Let $L$ be a finite dimensional
    $H_{m^2}(\zeta)$-simple Lie algebra over an algebraically closed field $F$ of characteristic $0$.
   Then there exist $C>0$ and $r\in \mathbb R$ such that
   $$C n^r (\dim L)^n \leqslant c_n^{H_{m^2}(\zeta)}(L) \leqslant (\dim L)^{n+1}\text{ for all }n\in \mathbb N.$$
   In particular, $\PIexp^{H_{m^2}(\zeta)}(L) = \dim L$ and the analog of Amitsur's conjecture holds
   for $L$.
  \end{theorem} 

First we need the following standard observation:

  \begin{lemma}\label{LemmaHStripeTheorem}
  Let $L$ be a finite dimensional $H$-module Lie algebra over a field of characteristic $0$.
    Let $\lambda \vdash n$, $n\in\mathbb N$.
    Suppose $\lambda_{(\dim L)+1} > 0$. Then $m(L, H,\lambda)= 0$.
  \end{lemma}
  \begin{proof}
    It is sufficient to prove that $e^{*}_{T_\lambda} f \in \Id^H(L)$
    for all $f\in V_n^H$.
Fix some basis of $L$.
Since polynomials are multilinear, it is
sufficient to substitute only basis elements.
 Note that
$e^{*}_{T_\lambda} = b_{T_\lambda} a_{T_\lambda}$
where $b_{T_\lambda}$ alternates the variables of each column
of $T_\lambda$. Hence if we make a substitution and $
e^{*}_{T_\lambda} f$ does not vanish, this implies
 that different basis elements
are substituted for the variables of each column.
But if $\lambda_{(\dim L)+1} > 0$, then the length of the first
 column is greater
than $\dim L$. Therefore,
 $e^{*}_{T_\lambda} f \in \Id^H(L)$.
  \end{proof} 
  
 Now we prove the existence of a polynomial $H$-non-identity with many alternations: 
\begin{lemma}\label{LemmaTaftSimpleLieAlt} Let $L$ be a finite dimensional
  non-semisimple  $H_{m^2}(\zeta)$-simple Lie algebra over an algebraically closed field $F$ of characteristic $0$. Let $\ell := \dim L^{(0)}$.
Then exists a number $r \in \mathbb N$ such that for every $n\geqslant \ell m r + 1$
there exist disjoint subsets $X_1$, \ldots, $X_{kr} \subseteq \lbrace x_1, \ldots, x_n
\rbrace$, $k := \left[\frac{n-1}{\ell m r}\right]$,
$|X_1| = \ldots = |X_{kr}|=\ell m$ and a polynomial $f \in V^{H_{m^2}(\zeta)}_n \backslash
\Id^{H_{m^2}(\zeta)}(L)$ alternating in the variables of each set $X_j$.
\end{lemma}
  \begin{proof}
  Since $L$ is not semisimple, Theorem~\ref{TheoremTaftSimpleNonSemiSimpleLieClassify} implies
  $L \cong L(B, 0)$ for some simple Lie algebra $B$, $\dim B = \ell$.
  We have the $\mathbb Z_m$-grading $L=\bigoplus_{k=0}^{m-1} L^{(k)}$
  (see Remark~\ref{RemarkTaftZmGrading})
  where $L^{(0)}$ can be identified with $B$.
  By Yu.\,P.~Razmyslov's theorem~\cite[Theorem~12.1]{Razmyslov},
  there exists $r\in\mathbb N$ such that for every $k\in\mathbb N$
  there exists a multilinear associative polynomial $$f_0=f_0(x_{11}, \ldots, x_{1\ell};\ldots; x_{kr,1}, \ldots, x_{kr, \ell})$$
  alternating in the variables of each set $\lbrace x_{i1}, \ldots, x_{i\ell} \rbrace$, $1\leqslant i \leqslant kr$, such that
  $f_0(\ad a_1, \ldots, \ad a_\ell;\ldots; \ad a_1, \ldots, \ad a_\ell)$
  is a nonzero scalar operator on $L$ for any basis $a_1, \ldots, a_\ell$ of $B$.
  Here $(\ad x)y := [x,y]$.
  
  Let $n\in\mathbb N$. Define $k := \left[ \frac{n-1}{\ell m r}\right]$. Choose a polynomial $f_0$
  as above alternating in $kr$ sets of $\ell$ variables and a polynomial $\tilde f_0$ 
  alternating in $mr$ sets of $\ell$ variables.
  Consider the Lie ${H_{m^2}(\zeta)}$-polynomial $$f_1:=\tilde f_0(\ad y_{11}, \ldots, \ad y_{1\ell};\ldots; \ad y_{mr,1}, \ldots, \ad y_{mr, \ell})f_2$$ where $$f_2=\left(\prod_{i=1}^m f_0(\ad (x_{11i}^{v^{i-1}}), \ldots, \ad (x_{1\ell i}^{v^{i-1}});\ldots; \ad (x_{kr,1,i}^{v^{i-1}}), \ldots, \ad (x_{kr, \ell,i}^{v^{i-1}}))\right)z.$$
  Let $b_1, \ldots, b_\ell$ be a basis of $L^{(m-1)}$.
  Then $v^{m-1} b_1, \ldots, v^{m-1} b_\ell$ is a basis of $L^{(0)}=B$.
  Hence $f_1$ does not vanish under the substitution $x_{jti}=v^{m-i} b_t$, $y_{jt}=v^{m-1} b_t$ and $z=\bar z$
  for any nonzero $\bar z \in B$. Fix some $\bar z$ and denote this substitution by $\Xi$.
  Let $b$ be the value of $f_1$ under $\Xi$.
  Consider $f_3 := \Alt_1 \ldots \Alt_{kr} f_1$
  where $\Alt_j$ is the operator of alternation in the variables of the set $X_j = \lbrace x_{jti} \mid 1\leqslant t \leqslant \ell,\ 1\leqslant i \leqslant m \rbrace$.
  Note that  $v^m=0$ implies that all the items where $x_{jti}$ is replaced with some $x_{jt'i'}$
  for $i' < i$, vanish. Hence all permutations in the alternations mix variables just in every set 
  $\lbrace x_{jti} \mid 1\leqslant t \leqslant \ell \rbrace$ for fixed $j$ and $i$.
  Since $f_1$ is alternating in the variables of these sets, the value of $f_3$ under $\Xi$
  equals $(\ell!)^{kmr} b \ne 0$.
  
  Note that $ k\ell m r+1 \leqslant n < \deg f_3 = (k+1)\ell mr+1$. We can expand $\tilde f_0$ and rewrite $f_3$ as a linear combination of 
  \begin{equation*}\begin{split}f_4 := \left[w_1, \left[w_2, \ldots, \left[w_{\ell m r}, \Alt_1 \ldots \Alt_{kr}
  \left(\prod_{i=1}^m f_0(\ad (x_{11i}^{v^{i-1}}), \ldots, \ad (x_{1\ell i}^{v^{i-1}}); \right.\right.\right.\right. \\ \left.\left.\left. \ldots; \ad (x_{kr,1,i}^{v^{i-1}}), \ldots, \ad (x_{kr, \ell,i}^{v^{i-1}}))\right)z\right]\ldots\right]\end{split}\end{equation*}
  where the variables $w_i$ are the variables $y_{jt}$ taken in some order depending on the item.
  Since $f_3 \notin \Id^{H_{m^2}(\zeta)}(L)$,
   one of the items $f_4$ does not vanish under $\Xi$.
  Then
  \begin{equation*}\begin{split}f = \left[w_{(k+1)\ell m r-n+2}, \left[w_{(k+1)\ell m r-n+3}, \ldots, \left[w_{\ell m r}, \Alt_1 \ldots \Alt_{kr}
  \left(\prod_{i=1}^m f_0(\ad (x_{11i}^{v^{i-1}}), \ldots, \ad (x_{1\ell i}^{v^{i-1}});\right.\right.\right.\right. \\ \left.\left.\left. \ldots; \ad (x_{kr,1,i}^{v^{i-1}}), \ldots, \ad (x_{kr, \ell,i}^{v^{i-1}}))\right)z\right]\ldots\right] \notin \Id^{H_{m^2}(\zeta)}(L).
  \end{split}\end{equation*}
  Now we notice that $\deg f = n$. If we rename the variables of $f$ to $x_1, x_2, \ldots, x_n$,
  then $f$ satisfies all the conditions of the lemma.
  \end{proof}
  \begin{proof}[Proof of Theorem~\ref{TheoremTaftSimpleLieHPIexpExists}]
  If $L$ is semisimple, then the assertion of the theorem is a consequence of~\cite[Example 10]{ASGordienko5}.
  
  Suppose $L$ is not semisimple. By~\cite[Lemma 1]{ASGordienko5},
  we still have the upper bound $c_n^{H_{m^2}(\zeta)}(L) \leqslant (\dim L)^{n+1}$.

    Let $r$ be the number from Lemma~\ref{LemmaTaftSimpleLieAlt}. Let $\ell :=\frac{\dim L}m$ and $k := \left[\frac{n-1}{\ell m r}\right]$.
  We claim that for every $n\in\mathbb N$ there exists $\lambda \vdash n$, $m(L, {H_{m^2}(\zeta)}, \lambda)\ne 0$,
  such that $\lambda_i \geqslant kr$ for all $1\leqslant i \leqslant \ell m$.
    Consider the polynomial $f$ from Lemma~\ref{LemmaTaftSimpleLieAlt}.
It is sufficient to prove that $e^*_{T_\lambda} f \notin \Id^{H_{m^2}(\zeta)}(L)$
for some tableau $T_\lambda$ of the desired shape $\lambda$.
It is known that $FS_n = \bigoplus_{\lambda,T_\lambda} FS_n e^{*}_{T_\lambda}$ where the summation
runs over the set of all standard tableaux $T_\lambda$,
$\lambda \vdash n$. Thus $$FS_n f = \sum_{\lambda,T_\lambda} FS_n e^{*}_{T_\lambda}f
\not\subseteq \Id^{H_{m^2}(\zeta)}(L)$$ and $e^{*}_{T_\lambda} f \notin \Id^{H_{m^2}(\zeta)}(L)$ for some $\lambda \vdash n$.
We claim that $\lambda$ is of the desired shape.
It is sufficient to prove that
$\lambda_{\ell m} \geqslant kr$, since
$\lambda_i \geqslant \lambda_{\ell m}$ for every $1 \leqslant i \leqslant \ell m$.
Each row of $T_\lambda$ includes numbers
of no more than one variable from each $X_i$,
since $e^{*}_{T_\lambda} = b_{T_\lambda} a_{T_\lambda}$
and $a_{T_\lambda}$ is symmetrizing the variables of each row.
Thus $\sum_{i=1}^{\ell m-1} \lambda_i \leqslant kr(\ell m-1) + (n-k\ell m r) = n-kr$.
 Lemma~\ref{LemmaHStripeTheorem} implies that if $\lambda \vdash n$ and $\lambda_{\ell m+1} > 0$, then $m(L, {H_{m^2}(\zeta)}, \lambda)=0$.
Therefore
$\lambda_{\ell m} \geqslant kr$.

The Young diagram~$D_\lambda$ contains
the rectangular subdiagram~$D_\mu$, $\mu=(\underbrace{kr, \ldots, kr}_{\ell m})$.
The branching rule for $S_n$ implies that if we consider the restriction of
$S_n$-action on $M(\lambda)$ to $S_{n-1}$, then
$M(\lambda)$ becomes the direct sum of all non-isomorphic
$FS_{n-1}$-modules $M(\nu)$, $\nu \vdash (n-1)$, where each $D_\nu$ is obtained
from $D_\lambda$ by deleting one box. In particular,
$\dim M(\nu) \leqslant \dim M(\lambda)$.
Applying the rule $(n-k\ell m r)$ times, we obtain $\dim M(\mu) \leqslant \dim M(\lambda)$.
By the hook formula, $$\dim M(\mu) = \frac{(k\ell m r)!}{\prod_{i,j} h_{ij}}$$
where $h_{ij}$ is the length of the hook with edge in $(i, j)$.
By Stirling formula,
$$c_n^{H_{m^2}(\zeta)}(L)\geqslant \dim M(\lambda) \geqslant \dim M(\mu) \geqslant \frac{(k\ell m r)!}{((kr+\ell m)!)^{\ell m}}
\sim $$ $$\frac{
\sqrt{2\pi (k\ell m r)} \left(\frac{k\ell m r}{e}\right)^{k\ell m r}
}
{
\left(\sqrt{2\pi (k r +\ell m)}
\left(\frac{k r +\ell m}{e}\right)^{k r +\ell m}\right)^{\ell m}
} \sim C_1 k^{r_1} (\ell m)^{k\ell m r}$$
for some constants $C_1 > 0$, $r_1 \in \mathbb Q$,
as $k \to \infty$. (We write $f\sim g$ if $\lim \frac{f}{g} = 1$.)
Since $k = \left[\frac{n-1}{\ell m r}\right]$,
this gives the lower bound.
\end{proof}

\end{document}